\documentclass[a4paper]{amsart}

\usepackage{lmodern,microtype}
\usepackage{amsmath,amstext,amssymb,mathrsfs,amscd,amsthm,indentfirst}
\usepackage{amsfonts}
\usepackage[dvipsnames,svgnames,x11names,hyperref]{xcolor}

\usepackage[pagebackref,colorlinks,citecolor=Mahogany,linkcolor=Mahogany,urlcolor=Mahogany,filecolor=Mahogany]{hyperref}

\usepackage{booktabs}
\usepackage{verbatim}
\usepackage{enumerate}
\usepackage{graphicx}
\usepackage{multicol}
\usepackage[utf8]{inputenc}
\usepackage{textcomp}
\usepackage{color}
\usepackage{mathtools}
\usepackage{multirow}
\usepackage{blindtext}
\usepackage{enumitem}
\usepackage{subcaption}
\usepackage{wrapfig}
\usepackage{float}
\usepackage{pst-node}
\usepackage{mathdots}

\usepackage{tikz, tikz-cd}
\usetikzlibrary{matrix,calc,positioning,arrows,decorations.pathreplacing,decorations.markings,patterns}
\tikzset{->-/.style={decoration={
			markings,
			mark=at position #1 with {\arrow{>}}},postaction={decorate}}}
\tikzset{-<-/.style={decoration={
					markings,
					mark=at position #1 with {\arrow{<}}},postaction={decorate}}}

\DeclareMathOperator{\map}{map}
\DeclareMathOperator{\mapd}{\map_{\partial}}
\DeclareMathOperator{\coker}{coker}
\DeclareMathOperator{\Diff}{Diff}
\DeclareMathOperator{\Diffd}{\Diff_{\partial}}
\DeclareMathOperator{\ev}{ev}

\DeclareMathOperator{\Aut}{Aut}
\DeclareMathOperator{\Inn}{Inn}
\DeclareMathOperator{\Out}{Out}

\DeclareMathOperator{\Sym}{Sym}
\DeclareMathOperator{\glnz}{GL_n(\Z)}
\DeclareMathOperator{\GL}{GL}
\DeclareMathOperator{\Sp}{Sp}
\DeclareMathOperator{\sptg}{\Sp_{2g}(\Z)}

\DeclareMathOperator{\im}{im}
\DeclareMathOperator{\id}{id}

\DeclareMathOperator{\Q}{\mathbb{Q}}
\DeclareMathOperator{\Z}{\mathbb{Z}}
\DeclareMathOperator{\A}{\mathcal{A}}

\DeclareMathOperator{\F}{\mathcal{F}}
\DeclareMathOperator{\R}{\mathbb{R}}

\DeclareMathOperator{\Tot}{Tot}
\DeclareMathOperator{\gr}{gr}
\DeclareMathOperator{\Th}{Th}

\DeclareMathOperator{\MCG}{MCG}

\mathchardef\ordinarycolon\mathcode`\:
\mathcode`\:=\string"8000
\begingroup \catcode`\:=\active
  \gdef:{\mathrel{\mathop\ordinarycolon}}
\endgroup

\usepackage{amsthm}
\theoremstyle{plain}
\newtheorem{theorem}{Theorem}[section]

\newtheorem{proposition}[theorem]{Proposition}

\newtheorem{conjecture}[theorem]{Conjecture}
\newtheorem{lemma}[theorem]{Lemma}
\newtheorem*{lemma*}{Lemma}

\newtheorem{corollary}[theorem]{Corollary}
\newtheorem*{corollary*}{Corollary}

\theoremstyle{definition}
\newtheorem{definition}[theorem]{Definition}
\newtheorem{question}[theorem]{Question}

\newtheorem{notation}[theorem]{Notation}
\newtheorem{example}[theorem]{Example}
\newtheorem{claim}[theorem]{Claim}
\newtheorem*{claim*}{Claim}
\newtheorem*{comment*}{Comment}

\theoremstyle{remark}

\newtheorem{remark}[theorem]{Remark}
\newtheorem*{remark*}{Remark}

\numberwithin{equation}{section}

\title[]{Cohomology of configuration spaces of surfaces as mapping class group representations}

\author{Andreas Stavrou}
\email{as2558@cam.ac.uk}
\address{Centre for Mathematical Sciences, Wilberforce Road, Cambridge CB3 0WB, UK}

\date{\today}
\begin{document}
\maketitle
\begin{abstract}
    We express the rational cohomology of the unordered configuration space of a compact oriented manifold as a representation of its mapping class group in terms of a weight-decomposition of the rational cohomology of the mapping space from the manifold to a sphere. We apply this to the case of a compact oriented surface with one boundary component and explicitly compute the rational cohomology of its unordered configuration space as a representation of its mapping class group. In particular, this representation is not symplectic, but has trivial action of the second Johnson filtration subgroup of the mapping class group. 
\end{abstract}
\section{Introduction}\label{sec:Introduction}
The \textit{ordered configuration space} of a manifold $M$ is the space $$F_n(M)=\{(x_1,...,x_n)\in \mathring{M}^n:x_i\neq x_j \text{ if } i\neq j\}$$ of $n$ distinct points in the \textit{interior} of $M$, topologised as a subspace of $M^n$.
The symmetric group $\mathfrak{S}_n$ acts on $F_n(M)$ by permuting the coordinates, and the \textit{(unordered) configuration space} of $M$ is the quotient $$C_n(M)=F_n(M)/\mathfrak{S}_n.$$ Henceforth, we will assume that the manifold $M$ is \textit{compact, connected} and \textit{oriented} of dimension $d$, possibly with boundary. Furthermore, all (co)homology will be with $\Q$ coefficients unless explicitly stated. 

The \textit{diffeomorphism group} $\Diffd(M)$ is the group of orientation preserving diffeomorphisms of $M$ pointwise preserving a neighbourhood of the boundary $\partial M$, with the Whitney $C^{\infty}$-topology. The \textit{mapping class group} $\MCG(M)$ of $M$ is the group of connected components $\pi_0(\Diffd(M))$, i.e. diffeomorphisms up to isotopy.

The group $\Diffd(M)$ acts naturally on $C_n(M)$ giving rise to an action of $\MCG(M)$ on the homology $H_*(C_n(M))$. We are interested in
\begin{question}
What is $H_*(C_n(M))$ as a $\MCG(M)$-representation?
\label{qu:thequestion}
\end{question}

\subsection{General manifolds}
We reduce Question \ref{qu:thequestion} to computing the rational cohomology of the space $\mapd(M,S^{d+2m}_{\Q})$ of continuous maps of pairs $(M,\partial M)\to (S^{d+2m}_{\Q},*)$. Here $S^{d+2m}_{\Q}$ is the rationalisation of the sphere $S^{d+2m}$. We will require the computation of two additional pieces of structure on $H^*(\mapd(M,S^{d+2m}_{\Q}))$:
\begin{enumerate}
    \item the $\MCG(M)$ action arising from the natural precomposition action of $\Diffd(M)$ on $\mapd(M,S^{d+2m}_{\Q})$ and
    
    \item the \textit{sphere action}, i.e. the action of the group
\begin{equation*}
    (\Q^*,\times)=\pi_0(\map_*(S^{d+2m}_{\Q},S^{d+2m}_{\Q}), \circ)^{\times}
\end{equation*}
on $H^*(\mapd(M,S^{d+2m}_{\Q}))$ arising from postcomposition of maps.
\end{enumerate}
For $k\in\Z$, the $k$-weightspace, $V^{(k)}$, of a $\Q^*$-representation $V$ is the subspace of $V$ on which every $q\in \Q^*$ acts via scalar multiplication by $q^k$.

\begin{theorem}[Theorem \ref{thm:maintheoremforgeneralmanifolds}]\label{thm:introgeneralmanifolds}
Let $d=\dim M$ and $m\ge 1$. Under the sphere action, $H^{*}(\map_{\partial}(M,S^{d+2m}_{\Q}))$ splits as a direct sum of its weightspaces and  there is a $\MCG(M)$-equivariant isomorphism
\begin{equation}
    H^j(C_k(M))\cong \widetilde{H}^{j+2mk}(\map_{\partial}(M,S^{d+2m}_{\Q}))^{(k)}.
\end{equation}
\end{theorem}

\subsection{Surfaces}
We will apply Theorem \ref{thm:introgeneralmanifolds} to answer Question \ref{qu:thequestion} for $M=\Sigma_{g,1}$, a compact oriented surface of genus $g\ge 1$ with one boundary component. In this case, the mapping class group $\Gamma_{g,1}=\MCG(\Sigma_{g,1})$ is a group with much interest for its connection with moduli spaces of Riemann surfaces, and thus algebraic geometry and theoretical physics. 

We took inspiration from previous works concerning Question \ref{qu:thequestion} for $\Sigma_{g,1}$. These include the computation of $H^*(C_n(\Sigma_{g,1}))$ as vector spaces in \cite{bodigheimer88rationalcohomologyofsurfaces}, a complete answer to Question \ref{qu:thequestion} for with $\Z_2$-coefficients in \cite{Bianchi2020}, and, most recently, some results with $\Q$-coefficients but \textit{ordered} configurations $F_n(\Sigma_{g,1})$ in \cite{bianchi2021mapping}. For closed surfaces $\Sigma_g$, a partial answer to Question \ref{qu:thequestion} is in \cite{pagaria2019cohomology}, with a proof that the $\MCG(\Sigma_g)$-action is not symplectic (see definition below) in \cite{looijenga2020torelli}. On the latter we comment in Section \ref{sec:closedsurfaces}.

\subsubsection{$\Gamma_{g,1}$-representations}
We fix $g\ge1$ and recall some representation theory of $\Gamma_{g,1}$. The action of $\Gamma_{g,1}$ on $H_{\Z}:=H_1(\Sigma_{g,1};\Z)\simeq\Z^{2g}$ preserves the intersection product $\lambda_{\Z}:\Lambda^2H_{\Z}\to\Z$ and thus induces the short exact sequence of groups \begin{equation}
    1\to I_{g,1}\to \Gamma_{g,1}\to \sptg\to 1,
\end{equation} 
where the \textit{symplectic group} $\sptg$ is the subgroup of the linear group $\GL_{2g}(\Z)$ preserving $\lambda_{\Z}$, and the \textit{Torelli group} $I_{g,1}$ is the kernel of this action. The rationalisation $H:=H_1(\Sigma_{g,1};\Q)\simeq \Q^{2g}$ is the \textit{standard symplectic representation} of $\Gamma_{g,1}$. The symplectic form $\lambda=\lambda_{\Z}\otimes \Q$ is $\Gamma_{g,1}$-invariant, so provides a $\Gamma_{g,1}$-equivariant isomorphism $H\cong H^{\vee}$, under which $\lambda$ dualises itself to an invariant tensor $\omega\in \Lambda^2H$. Any $\Gamma_{g,1}$-representation that factors through $\sptg$, or equivalently with trivial $I_{g,1}$-action, is called \textit{symplectic}. The standard representation $H$ is symplectic, and so are all the subquotients of its tensor powers.

We fix a \textit{standard symplectic basis} $\alpha_1,...,\alpha_{2g}\in \pi_1(\Sigma_{g,1})=\Z^{*2g}$, write $[-]:\pi_1(\Sigma_{g,1})\to H$ the abelianisation and obtain the standard symplectic basis $e_i:=[\alpha_i]$ of $H$. In this notation $\omega=e_1\wedge e_2+...+e_{2g-1}\wedge e_{2g}\in \Lambda^2H$. 

The \textit{content} (Definition \ref{def:content}) is the function $$c:\pi_1(\Sigma_{g,1})\to \Lambda^2H$$ defined on each word $w=w_1w_2...w_k\in\pi_1(\Sigma_{g,1})$, in the alphabet $\{\alpha_i\}$, by
\begin{equation}
    c(w)=\sum_{\mathclap{1\le i<j\le k}} \hspace{6pt} [w_i]\wedge [w_j].
\end{equation}
The content depends on the choice of basis. Some examples of how it evaluates are: 1. $c(\alpha_i^{\pm})=0$ for $i=1,...,2g$; 2. on commutators , $c([a,b])=c(aba^{-1}b^{-1})=2[a]\wedge [b]$ for  $a,b\in \pi_1(\Sigma_{g,1})$, and 3. on the boundary word $c([\alpha_1,\alpha_2]...[\alpha_{2g-1},\alpha_{2g}])=2(e_1\wedge e_2+...+e_{2g-1}\wedge e_{2g})=2\omega$ (Proposition \ref{prop:propertiesofcontent}).

We use the content $c$ to define function
    $\xi:\Gamma_{g,1}\to \hom(H,\Lambda^2H)$ (Definition \ref{def:xi}) by setting  
    $$\xi(\phi)(e_i)=c(\phi(\alpha_i))$$
    for each $\phi\in \Gamma_{g,1}$ and $i=1,...,2g$, and extending linearly on $H$. The function $\xi$ is \textit{not} a group homomorphism, but rather a crossed homomorphism that agrees on $I_{g,1}$ with the Johnson homomorphism $\tau$ from \cite{JohnsonAnAbelianQuotient} (see proof of Proposition \ref{prop:torelliactionnontrivial}).
    
\subsubsection{$\Gamma_{g,1}$-algebras}    
All algebras will be over $\Q$. For a group $G$, a $G$-algebra is an algebra with an action of $G$ by algebra automorphisms; usually $G=\Gamma_{g,1}$.

The algebras $\Lambda[x_1,...,x_{2g}]$ and $\Q[y_1,...,y_{2g}]$ are naturally $\sptg$-algebras by treating the $x_i$ and $y_i$ as standard symplectic bases.
Viewing the function $\xi$ as landing in $\hom(\Q\{y_1,...,y_{2g}\},\Lambda^2\{x_1,...,x_{2g}\})$, we define a new $\Gamma_{g,1}$-structure on the product algebra
$\Q[y_1,...,y_{2g}]\otimes \Lambda[x_1,...,x_{2g}]$
by insisting that for $\phi\in \Gamma_{g,1}$ and $i=1,...,2g$,
$$\phi\cdot x_i=\phi\cdot_{\sptg} x_i$$
and
$$\phi\cdot y_i=\phi\cdot_{\sptg} y_i+\xi(\phi)(y_i).$$ 
    
Finally, suppose $T_g$ is another $\Gamma_{g,1}$-algebra with a compatible $\Lambda[x_1,...,x_{2g}]$-module\footnote{In this paper, \textit{module} will mean \textit{left module}.} structure. We specify a new $\Gamma_{g,1}$-structure on the $\Q$-algebra $\Q[y_1,...,y_{2g}]\otimes T_g$, referred to as \textit{the Johnson action on the $y_i$} (see Definition \ref{def:algerbaswithJohnsonaction}) using the natural $\Gamma_{g,1}$-structure on the right hand side of $$\Q[y_1,...,y_{2g}]\otimes T_g=(\Q[y_1,...,y_{2g}]\otimes \Lambda[x_1,...,x_{2g}])\otimes_{\Lambda[x_1,...,x_{2g}]} T_g.$$
The homology algebra $H^*(\Lambda[x_1,...,x_{2g},v], d)$, with differential $d$ given by $d(x_i)=0$ and $d(v)=2\omega$, is such an example of a $\Gamma_{g,1}$-algebra with a compatible $\Lambda[x_1,...,x_{2g}]$-module structure (Lemma \ref{lem:structureofhomology}). 

\begin{notation} For an ungraded vector space $V$, $V[i_1,...,i_k]$ denotes the graded vector space isomorphic to $V$ concentrated in degree $(i_1,...,i_k)$. For a  graded vector space $V$, or a set $A$ of  graded generators, we define the graded algebras
\begin{enumerate}
    \item $\Lambda(V)$, resp. $\Lambda[A]$, the free \textit{exterior} algebra;
    \item $\Q(V)$, resp. $\Q[A]$,  the free \textit{polynomial} algebra;
    \item $S(V)$, resp. $S[A]$,  the free \textit{graded} commutative algebra.  
\end{enumerate}
A grading in bracket will not contribute a sign to graded commutativity.
\end{notation}

\begin{theorem}[Theorem \ref{thm:maintheoremsurfaces}]\label{thm:intromaintheoremsurfaces}
For $i,n\ge 0$, the $\Gamma_{g,1}$-representation $H^i(C_n(\Sigma_{g,1}))$
is isomorphic to the bidegree $(i,(n))$ part of the bigraded $\Gamma_{g,1}$-algebra
\begin{equation*}
    \Q[y_1,...,y_{2g},w]\otimes H^{*,(*)}(\Lambda[x_1,...,x_{2g},v], d)
\end{equation*}
with
\begin{itemize}
    \item bidegrees $|x_i|=(1,(1))$, $|y_i|=(2,(2))$,
    $|w|=(0,(1))$ and $|v|=(1,(2))$;
    
    \item differential $d$ given by 
\begin{equation*}
    d(v)=2\omega:=2(x_1x_2+...+x_{2g-1}x_{2g}),
\end{equation*}
    vanishing on the $x_i$, and extended using the Leibniz rule;
    
    \item the trivial $\Gamma_{g,1}$-action on $w$ and $v$, the symplectic action on the $x_i$ and the Johnson action on the $y_i$.
\end{itemize}
\end{theorem}
\begin{remark}
Even though we will make heavy use of the cup-product in Sections \ref{sec:mapsfromwedges} and \ref{sec:mapsrelboundary}, we make no claims of computing the ring structure of $H^*(C_n(M))$. (But see Conjecture \ref{conj:ringstructure}).
\end{remark}

\begin{remark}
We give an alternative basis-free description of the algebra in Theorem \ref{thm:intromaintheoremsurfaces}. Using $\xi$, we define the \textit{Johnson representation} $J$ of $\Gamma_{g,1}$ (Proposition \ref{prop:JohnsonRep}), fitting in an extension of $\Gamma_{g,1}$-representations
\begin{equation}
\begin{tikzcd}
     0\rar & \Lambda^2H\rar["i"] & J\rar & H\rar& 0.
\end{tikzcd}
    \label{eq:introdefJ}
\end{equation}
The bigraded commutative $\Gamma_{g,1}$-algebra
$R^{*,(*)}$ is the quotient 
\begin{equation}
    R^{*,*}=\Q\big(J[2,(2)]\oplus \Q[0,(1)]\big)\otimes \Lambda\big(H[1,(1)]\oplus  \Q[1,(2)]\big)/I,
\end{equation}
where $\Q$ is the trivial $\Gamma_{g,1}$-representation and $I$ is the ideal
\begin{equation}
    I=\Big(z-i(z):z\in \Lambda^2\big(H[1,(1)]\big)\Big),
\end{equation}
and differential $d_R$ given by
$d_R(1[1,(2)])=2\omega \in \Lambda^2\big(H[1,(1)]\big)$
on the generator $1[1,(2)]\in\Q[1,(2)]$,
vanishing on the other generating vector spaces and extended by the Leibniz rule. The isomorphism from Theorem \ref{thm:intromaintheoremsurfaces} becomes $H^i(C_n(\Sigma_{g,1}))\cong H^{i,(n)}(R,d_R)$.
\end{remark}

Finally, we can filter $\Gamma_{g,1}$ by the \textit{Johnson filtration} 
$$...\subset J(i+i)\subset J(i)\subset... \subset J(0)=\Gamma_{g,1}$$ defined by
\begin{equation}
    J(i)=\ker\Big(\Gamma_{g,1}\curvearrowright \pi_1(\Sigma_{g,1})/\pi_1(\Sigma_{g,1})^{(i)}\Big)
\label{eq:JohnsonFiltration}
\end{equation}
for $i\ge 0$, 
where $\pi_1(\Sigma_{g,1})^{(i)}$ is the $i$-th lower central subgroup of $\pi_1(\Sigma_{g,1})$.
Clearly $J(1)=I_{g,1}$.

\begin{corollary}[Corollary \ref{cor:actionofJ2trivial}, Proposition \ref{prop:torelliactionnontrivial}]
The $\Gamma_{g,1}$-action  on $H^*(C_n(\Sigma_{g,1}))$ is non-trivial on $I_{g,1}$, if $g\ge 2$, but is trivial on the second Johnson filtration subgroup $J(2)$ for all $g\ge 0$.
\end{corollary}

\subsection{Acknowledgments} The author is indebted to his PhD supervisor Oscar Randal-Williams for suggesting the problem, and his continuous discussions, support and motivation over the past two years.

\section{The group $\Q^*$, sphere actions, and weights}\label{sec:generalities}
Let $\Q^*$ denote the group of non-zero rationals with multiplication. For a simply connected space $X$, a map $r:X\to X_{\Q}$ is a \textit{rationalisation} if it induces rationalisations on homotopy or, equivalently, homology groups. The rationalisation $X_{\Q}$ is uniquely defined up to homotopy equivalence.

\begin{definition}
    A based $\Q^*$-action on a rational sphere $(S^{n}_{\Q},*)$, $n\ge 1$, by homeomorphisms is called a \textit{sphere action} if the induced action on $\pi_n(S^n_{\Q})=\Q$ is by multiplication by $q$ for every $q\in \Q^*$. 
\end{definition}
\begin{remark}
Sphere actions exist (see Section \ref{sec:oddsphereseilenbergmaclane} for an example), but we will not specify one yet.
In fact, for any $d,n\ge 1$, since $S^{d}\wedge S^{n}_{\Q}\simeq_{\Q}S^{d+n}_{\Q}$, specifying a sphere action on $S^{n}_{\Q}$ induces a sphere action on $S^{d}\wedge S^{n}_{\Q}$.
\end{remark}

\begin{definition}[Weightspaces]
For a $\Q^{*}$-representation $V$ and $k\in \Z$, the \textit{$k$-weightspace} $V^{(k)}$ is the subspace of $V$ on which each $q\in \Q^{*}$ acts by multiplication by $q^k$. We say that $\Q^{*}$ acts on $V$ \textit{purely by weights} if $V$ splits as the direct sum of its weightspaces.  
\end{definition}

\begin{definition}[Weighted algebras]
A \textit{weighting} on a (differential) multi-graded commutative algebra $A^{*,...,*}$ is an extra grading denoted by $A^{*,...,*,(*)}$ that is compatible with the product of $A$, but \textit{does not} contribute a Koszul sign in the graded commutativity rule (nor in the Leibniz rule). If $A$ has a weighting, we call $A$ \textit{weighted}.  

If $V$ is a weighted multigraded vector space, then $S(V)$ is the free graded commutativity algebra on $V$ that ignores the weight in the graded commutativity.

Once defined, the weighting index will be usually omitted to alleviate notation.
\label{def:weightedalgebras}
\end{definition}

\begin{remark}
A multi-graded algebra $A^{*,...,*}$ on which $\Q^{*}$ acts purely by weights inherits a \textit{weighting} $A^{*,...,*,(*)}$. Cohomology of a space defined functorially in $S^{n}_{\Q}$ will usually be naturally weighted by the sphere action.
\end{remark}

\section{Configuration spaces in terms of mapping spaces, equivariantly}\label{sec:equivariantscanning}
The cohomology of mapping spaces $H^*(\mapd(M,X))$ is covariantly functorial in the manifold $(M,\partial M)$ with boundary preserving maps and contravariant in the based space $(X,x_0)$ with based maps. 
Therefore, the algebra $H^*(\mapd(M,S^{d+2m}_{\Q}))$ inherits commuting actions of $\MCG(M)$ by precomposition of maps and of $\Q^*$ by postcomposition with the sphere action. The aim of this section is proving

\begin{theorem} \label{thm:maintheoremforgeneralmanifolds}
For any $m\ge 1$, the sphere action on $H^{*}(\map_{\partial}(M,S^{d+2m}_{\Q}))$ is purely by weights, and there is a $\MCG(M)$-equivariant isomorphism
\begin{equation}
    H^i(C_k(M))\cong \widetilde{H}^{i+2mk}(\map_{\partial}(M,S^{d+2m}_{\Q}))^{(k)}.
\end{equation}
\end{theorem}
We interpolate the claimed isomorphism using the auxiliary space $\Gamma_{\partial}(M;S^{2m}_{\Q})$ which we define now.
Let $\tau^+ M$ be the fibrewise one point compactification of the tangent bundle of the manifold $M$. For a based space $(X,x_0)$, let $\tau^+ X:=\tau^+ M\wedge_f X$ be the fibrewise smash product of $(\tau^+_xM,0)$ with $(X,x_0)$. This has a canonical section $s_0$ defined by the smash point. Denote by $\Gamma_{\partial}(M;X)$ the space of continuous sections of $\tau^+_X$ defined on $M$ that agree with $s_0$ on $\partial M$. The diffeomorphism group $\Diffd(M)$ acts on $\Gamma_{\partial}(M;X)$ via bundle maps on $\tau M$. For $X=S^{2m}_{\Q}$, there is also an induced sphere action on $\Gamma_{\partial}(M;S^{2m}_{\Q})$ which commutes with the $\Diffd(M)$ action, just like in the case of mapping spaces.

Our argument is divided in two steps, each giving one of the $\MCG(M)$-equivariant isomorphisms 
\begin{equation}
    H^i(C_k(M))\cong_1\widetilde{H}^{i+2mk}\big(\Gamma_{\partial}(M;S^{2m}_{\Q})\big)^{(k)}\cong_2 \widetilde{H}^{i+2mk}(\map_{\partial}(M,S^{d+2m}_{\Q}))^{(k)}.
    \label{eq:theintermediateisomorphism}
\end{equation}
for $i,k,m\ge 1$. Theorem \ref{thm:weighttheorem} proves isomorphism 1 and Theorem \ref{thm:trivialisationequivariance} proves isomorphism 2.

\subsection{From configuration spaces to section spaces}\label{sec:sectionstoconfigurations}
In this proof, we give an equivariant version of an argument from \cite{BODIGHEIMER1989111generalconfigurations}, that uses the Thom space of a bundle over $C_n(M)$, by applying an equivariant stable splitting result from \cite{Tillmanthorpe2014}.

\begin{theorem}\label{thm:weighttheorem}
For $m\ge 1$, the sphere action on $H^*(\Gamma_{\partial}(M;S^{2m}_{\Q}))$ is purely by weights and there is a $\MCG(M)$-equivariant isomorphism
\begin{equation*}
    H^i(C_k(M))\cong\widetilde{H}^{i+2mk}\big(\Gamma_{\partial}(M;S^{2m}_{\Q})\big)^{(k)}.
\end{equation*}
\end{theorem}
\begin{proof}
We apply an equivariant stable splitting theorem from \cite{Tillmanthorpe2014}. To state it, we define the \textit{configuration space of $M$ with labels in the based space $(X,x_0)$}
\begin{equation}
    C(M;X)=\Big(\bigsqcup_{n\ge 1}F_n(M)\times_{\mathfrak{S}_n} X^n\Big)\big/ \sim
\end{equation}
where $(z_1,...,z_n;x_1,...,x_n)\sim (z_1,...,z_{n-1};x_1,...,x_{n-1})$ if $x_n=x_0$, and its filtration $C_{\le k}(M;X)\subset C(M;X)$ by configurations with at most $k$ configuration points. Denote by $D_k(M;X):=C_{\le k}(M;X)/C_{\le k-1}(M;X)$ the filtration quotients, which are naturally given by the formula
\begin{equation}
    D_k(M;X)=(F_k(M)_+\wedge X^{\wedge k})/\mathfrak{S}_k.
    \label{eq:definitionDk}
\end{equation}

Let $G$ be a group acting by based homeomorphisms on $X$. A special case of corollary 4.2 from \cite{Tillmanthorpe2014} is that there is a $\MCG(M)\times G$-equivariant isomorphism
\begin{equation}
    \widetilde{H}_*(\Gamma_{\partial}(M;X))\cong \bigoplus_{k\ge 1} \widetilde{H}_*(D_k(M;X)).
    \label{eq:tillmanthorpetheorem}
\end{equation}

We apply this to $X=S^{2m}_{\Q}$ and $G=\Q^*$ with a sphere action and obtain a $\MCG(M)\times \Q^*$-equivariant isomorphism
\begin{equation}
    \widetilde{H}_*(\Gamma_{\partial}(M;S^{2m}_{\Q}))\cong \bigoplus_{k\ge 1} \widetilde{H}_*(D_k(M;S^{2m}_{\Q})).
    \label{eq:equivariantsplittingQ}
\end{equation}

It is easy to see from formula \eqref{eq:definitionDk} that a rationalisation map $r:S^{2m}\to S^{2m}_{\Q}$ induces a $\Diffd(M)$-equivariant $\Q$-homology isomorphism $D(r):D_{ k}(M;S^{2m})\to D_{k}(M;S^{2m}_{\Q})$. For each $n\in \Q^*\cap\Z$, we construct a based map $f_n:S^{2m}\to S^{2m}$ of degree $n$ by taking the standard degree $n$ map on $S^1$ and doubly suspending it $2m-1$ times. Then $f_n$ is of degree $n$, it fixes the two suspension points $0,\infty\in S^{2m}$,  and $f_n^{-1}(0)=\{0\}$ and  $f_n^{-1}(\infty)=\{\infty\}$. The diagram
\begin{equation}
    \begin{tikzcd}
        D_{ k}(M;S^{2m})\rar["D(r_m)"]\dar["F_n"] & D_{k}(M;S^{2m}_{\Q})\dar["n_D"]\\
        D_{ k}(M;S^{2m})\rar["D(r_m)"] & D_{k}(M;S^{2m}_{\Q})
    \end{tikzcd},
    \label{eq:diagramDk}
\end{equation}
with $n_D$ the sphere action of $n\in \Q^*$ and $F_n$ induced from $f_n$, commutes up to homotopy.
\begin{claim}\label{cl:claim}
For each $k\ge 0$, there is an $\MCG(M)$-equivariant isomorphism 
\begin{equation}
    \widetilde{H}^i(D_k(M;S^{2m}))\cong H^{i-2mk}(C_k(M))
\end{equation}
and, for each $n\in \Z^*$, the induced map $F_n^*$ on $\widetilde{H}^*(D_k(M;S^{2m}))$ is by multiplication by $n^k$.
\end{claim}
Assuming Claim \ref{cl:claim}, we conclude that $\widetilde{H}^i(D_k(M;S^{2m}_{\Q}))\cong H^{i-2mk}(C_k(M))$, $\MCG(M)$-equivariantly. Then the direct sum in the RHS of isomorphism \eqref{eq:equivariantsplittingQ} is finite, so the isomorphism dualises to
\begin{equation}
    \widetilde{H}^i(\Gamma_{\partial}(M;S^{2m}_{\Q}))\cong \bigoplus_{k\ge 1} {H}^{i-2mk}(C_k(M))
\end{equation}
for $i\ge 0$, 
where, from the second part of Claim \ref{cl:claim}, the $\Q^*$-action on each $H^{i-2mk}(C_k(M))$ is by weight $k$. Taking $k$-weightspaces gives the desired result. 

It now remains to prove Claim \ref{cl:claim}.
As in \cite{bodigheimer88rationalcohomologyofsurfaces}, we define the vector bundle $$\eta_{k,2m}:F_k(M)\times_{\mathfrak{S}_k} (\R^{2m})^k\to C_k(M).$$
By identifying the label-space $S^{2m}$ with the one-point compactification $(\R^{2m})^+$,
we obtain a $\Diffd(M)$-equivariant homeomorphism
\begin{equation}
    h:D_k(M;S^{2m})\to \Th(\eta_{k,2m})
\end{equation}
to the Thom space of $\eta_{k,2m}$.
Denote by $E$ the total space of $\eta_{k,2m}$ and by $E^{\#}$ the complement of the $0$-section.

The bundle $\eta_{k,2m}$ is orientable since, in the definition of $\eta_{k,2m}$, the symmetric group $\mathfrak{S}_k$ permutes copies of the even dimensional $\R^{2m}$. Fixing an orientation on $\eta_{k,2m}$, we obtain a \textit{Thom class} $u_{E}\in H^{2mk}(E,E^{\#})$. The Thom isomorphism is given by 
\begin{align*}
    \tau:H^i(C_k(M))&\longrightarrow H^{i+2mk}(E,E^{\#})\cong \widetilde{H}^{i+2mk}(\Th(\eta_{k,2m}))\\
    x&\longmapsto u_{E}\cup \eta_{k,2m}^*(x).
\end{align*}
To see that $\tau$ is $\MCG(M)$-equivariant, it is enough to check that $u_E$ is $\MCG(M)$-invariant. But each $\phi\in\Diffd(M)$ can be isotoped so that it fixes an open $U\subset M$ pointwise. Picking a configuration $s\in C_k(M)$ entirely contained in $U$, $\phi$ acts trivially on the fibre $E_s$ and in particular it preserves orientation. Consequently, $\phi$ fixes the restriction $u_E|_s$ on the fibre at $s$, and thus it fixes $u_E$. We conclude that
\begin{equation}
    h^*\circ \tau:H^i(C_k(M))\to \widetilde{H}^{i+2mk}(D_k(M;S^{2m}))
\end{equation}
is a $\MCG(M)$-equivariant isomorphism.

Finally, we show that that $g_n=h\circ F_n\circ h^{-1}$ acts on $H^*(\Th(\eta_{k,2m}))$ by multiplication by $n^k$. Observe that by the definition of the map $F_n$, the map $g_n$ restricts on $E=\Th(\eta_{k,2m})-\{\infty\}$ to a (non-linear) $(\R^{2m})^k$-bundle map that covers the identity on $C_k(M)$, fixes the $0$-section and maps each $E_s^{\#}\cong((\R^{2m})^k)^{\#}$, $s\in C_k(M)$, into itself. Furthermore, the map $g_n$ acts by degree $n^k$ on the smash product of $k$ spheres $(S^{2m})^{2k}$ that appears as $((\R^{2m})^k)^+\cong E_s^{+}$ in $\Th(\eta_{k,2m})$. Therefore, $g_n$ acts on each pair $(E_s,E_s^{\#})$ by degree $n^k$. So $g_n^*(u_{E})=n^k \cdot u_{E}$ and, since $g_n$ covers the identity on $C_k(M)$, then $g_n^*(u_{E}\cup \eta_{k,2m}^*(x))=n^k\cdot u_{E}\cup \eta_{k,2m}^*(x)$. We conclude that $g_n$ acts on $\widetilde{H}^*(\Th(\eta_{k,2m}))$ by multiplication by $n^k$ and, thus, so does $F_n$ on $\widetilde{H}^*(D_k(M;S^{2m}))$. This concludes the proof of the Claim \ref{cl:claim} and of the theorem.
\end{proof}

\subsection{From section spaces to mapping spaces}\label{sec:sectionstomaps}
We prove isomorphism 2 from \eqref{eq:theintermediateisomorphism} to ``untwist'' the section space into a mapping space. 

\begin{theorem}\label{thm:trivialisationequivariance}
For $n\ge 2$, there is an $\MCG(M)\times \Q^{*}$-algebra isomorphism
\begin{equation}
    H^*(\Gamma_{\partial}(M;S^{n}_{\Q})) \cong H^*(\map_{\partial}(M;S^{n+d}_{\Q})).
    \label{eq:trivialisationequivariance}
\end{equation}

\end{theorem}
\begin{proof}
We fix: 1. an integer $0\le k\le n-1$ so that $d+k$ is odd, 2. a model for $S^k_{\Q}$, 3. a model for $S^{n-k}_{\Q}$ with a sphere action, and 4. $S^d\wedge S^{k}_{\Q}$ as our model for $S^{d+k}_{\Q}$. \footnote{Both sides of isomorphism \eqref{eq:trivialisationequivariance} are $\MCG(M)\times\Q^*$-equivariantly independent of the choice of model for the respective rational sphere.}

Let $E^+\wedge_f S^n_{\Q}$ be the fibrewise one-point-compactification of an oriented vector bundle $E^d\to B$ fibrewise-smashed with $S^n_{\Q}$. By decomposing $S^n_{\Q}=S^{k}_{\Q}\wedge S^{n-k}_{\Q}$, we observe that $$E^+\wedge_f S^n_{\Q}= E^+\wedge_fS^{k}_{\Q}\wedge_fS^{n-k}_{\Q}$$ 
is, in particular, an oriented $S^d\wedge S^{n}_{\Q}$-fibration-with-a-section fibrewise-smashed with $S^{n-k}_{\Q}$. 
The universal oriented $S^{d+k}_{\Q}$-fibration-with-a-section is 
\begin{equation}
    \pi_S:E_{d+k}\longrightarrow BhAut_*^1(S^{d+k}_{\Q}),
\end{equation}
whose base is the classifying space of the monoid $hAut_*^1(S^{d+k}_{\Q})$ consisting of based degree-$1$ homotopy self-equivalences of $S^{d+k}_{\Q}$. Therefore $E^+\wedge_f S^n_{\Q}$ is pulled back from $E_{d+k}\wedge_fS^{n-k}_{\Q}$ via a classifying map $B\to BhAut_*^1(S^{d+k}_{\Q})$.

We apply this to the universal oriented rank-$d$ vector bundle $\pi_{SO}:\gamma_d\to BSO(d)$ and fix a classifying map $g:BSO(d)\to BhAut_*^1(S^{d+k}_{\Q})$ for $\gamma_d^+\wedge_f S^n_{\Q}$. We further fix a classifying map $\tau_M:M\to BSO(d)$ for the tangent bundle of $M$. In all, we obtain the following diagram with two pullback squares
\begin{equation}
    \begin{tikzcd}
    \tau^+M\wedge_fS^{n}_{\Q}\rar\dar["\pi_{M}"]\arrow[dr, phantom, "\lrcorner", very near start] & \gamma_d^+M\wedge_fS^{n}_{\Q}\rar\dar["\pi_{SO(d)}"]
    \arrow[dr, phantom, "\lrcorner", very near start]
    & E_{d+k}\wedge_fS^{n-k}\dar["\pi_{S}"]\\
    M\rar["\tau_M"] &
    BSO(d)\rar["g"] &
    BhAut_*^1(S^{d+k}_{\Q}).
    \end{tikzcd}
    \label{eq:classifyingbundles}
\end{equation}

For a fibre bundle $\pi:E\to B$ with section $s$ and a map $f:M\to B$, let us denote by $L_{\partial}(f,\pi)$ the space of lifts $l$ of $f$ along $\pi$ agreeing with $s$ on $\partial M$, i.e.
\begin{equation}
    L_{\partial}(f,\pi):=\Bigg\{\begin{tikzcd}[row sep=small]
        & E\dar["\pi"'] \\
        (M,\partial M)\rar["f"]\arrow[ur,dashed, "l"] &
        B\uar["s"', bend right]
    \end{tikzcd}: l|_{\partial M}=s\circ f\Bigg\}
\end{equation}
By definition, $\Gamma_{\partial}(M;S^{n})=L_{\partial}(\id_M,\pi_M)$. Furthermore, any homotopy $g\simeq_H f$ produces a homotopy equivalence \begin{equation}
    {L_H}:\begin{tikzcd}
    L_{\partial}(g,\pi)\rar["\sim"]& L_{\partial}(f,\pi)\end{tikzcd},
\end{equation}
and homotopic homotopies $H,H'$ produce homotopic maps $L_H\simeq L_{H'}$. Finally, from the universal property of pullbacks in the diagram \eqref{eq:classifyingbundles}, we can canonically identify
\begin{equation}
    \Gamma_{\partial}(M;S^{n}_{\Q})\cong L_{\partial}(\id_M,\pi_{M})\cong L_{\partial}(g\circ \tau_M,\pi_{S}).
\end{equation}

Now, in general, \begin{equation}
    hAut_*^1(S^{N}_{\Q})\simeq_{\Q}\Omega^k_1S^N_{\Q}\simeq 
    \begin{cases}
    * \text{ , for }N\text{ odd,}\\
    K(\Q,N-1) \text{ , for }N\text{ even,}
    \end{cases}
\end{equation}
so, in particular, $BhAut_*^1(S^{d+k}_{\Q})$ is contractible.
Therefore, $g\circ \tau_M$ is nullhomotopic, say via a null-homotopy $H$, so $L_{\partial}(g\circ \tau_M,\pi_{S})\simeq_{L_H} L_{\partial}(c_*,\pi_{S})$ where $c_*$ is the constant map. Lifts of the constant map are just maps to the fibre, i.e. $L_{\partial}(c_*,\pi_{S})=\map_{\partial}(M;S^{d+k}_{\Q}\wedge S^{n-k}_{\Q})$. In all, we obtain a homotopy equivalence
\begin{equation}
    \gamma_{H}:\begin{tikzcd}
        \Gamma_{\partial}(M;S^{n}_{\Q})\rar["\sim"]& \map_{\partial}(M;S^{d+k}_{\Q}\wedge S^{n-k}_{\Q})
    \end{tikzcd}
\end{equation}
and the induced map $\gamma_H^*$ is the desired isomorphism. Since all the constructions above leave the final $S^{n-k}_{\Q}$ coordinate untouched, the equivalence $\gamma_H$ commutes with the sphere action on $S^{n-k}_{\Q}$, and so $\gamma_H^*$ is $\Q^*$-equivariant. It remains to check that it is also $\Diff(M,\partial M)$-equivariant.

For $\psi\in \Diff(M,\partial M)$, write $d\psi:\tau M\to \tau M$ for the derivative bundle map. Then $\psi$ acts on $s\in \Gamma_{\partial}(M;S^{n}_{\Q})$, by $\psi*s=d\psi\circ s\circ \psi^{-1}$ and on $f\in\map_{\partial}(M;S^{d+n}_{\Q})$ by $\psi*f=f\circ\psi^{-1}$. We obtain the commuting diagram
\begin{center}
\begin{tikzcd}
\Gamma_{\partial}(M;S^{n}_{\Q})
    \rar["\cong"'] \dar["\psi*-"]\arrow[rr,"\gamma_{H\circ \psi}",bend left=15] &
L_{\partial}(g\circ \tau_M\circ\psi,\pi_{Aut^1_*}) 
    \rar["L_{H\circ \psi}","\simeq"'] \dar["-\circ \psi^{-1}"] &
\map_{\partial}(M;S^{d+k}_{\Q}\wedge S^{n-k}_{\Q}) 
    \dar["\psi * -"]\\
\Gamma_{\partial}(M;S^{n}_{\Q}) 
    \rar["\cong"]\arrow[rr,"\gamma_H"', bend right=15]&
L_{\partial}(g\circ \tau_M, \pi_{Aut^1_*})
    \rar["L_{H}","\simeq"']&
\map_{\partial}(M;S^{d+k}_{\Q}\wedge S^{n-k}_{\Q}).
\end{tikzcd}
\end{center}
But, again, $BhAut^1_*(S^{d+k}_{\Q})$ is contractible, so the two homotopies $H$ and $H\circ \psi$ are homotopic and thus give homotopic maps $\gamma_H\simeq\gamma_{H\circ \psi}$. So $\gamma_H^*=\gamma_{H\circ \psi}^*$ commutes with the action of $\psi$ on the two spaces. 
\end{proof}
\begin{remark}
In the case that $M$ is parallelisable, a choice of trivialisation of $\tau M$ provides a homeomorphism $\Gamma_{\partial}(M;S^{n}_{\Q}) \cong \map_{\partial}(M;S^{n+d}_{\Q})$, that gives rise to the isomorphism \eqref{eq:trivialisationequivariance}. However, even in that case the whole proof of Theorem \ref{thm:trivialisationequivariance} is still needed to prove the $\MCG(M)$-equivariance. This is because the trivialisation is not preserved by $\Diffd(M)$, so the homeomorphism is not $\Diffd(M)$-equivariant. 
\end{remark}

\subsection{Final comments}
Let $\Z^*\subset \Q^*$ be the submonoid of non-zero integers with multiplication.
We could define an \textit{integral sphere action} of $\Z^*$ on a sphere $S^{2m}$ so that each $n\in \Z^*$ acts by an appropriate degree $n$ map, and furthermore define weightspaces for $\Z^*$-representations analogously to $\Q^*$.

\begin{corollary}\label{cor:maintheoremgeneralmanifoldsIntegers}
For any $m\ge 1$, the integral sphere action on $H^{*}(\map_{\partial}(M,S^{d+2m}))$ is purely by weights, and there is a $\MCG(M)$-equivariant isomorphism
\begin{equation*}
    H^i(C_k(M))\cong H^{i+2mk}(\map_{\partial}(M,S^{d+2m}))^{(k)}.
\end{equation*}
\end{corollary}
\begin{proof}
The rationalisation $S^{2m}\to S^{2m}_{\Q}$ induces a $\Diffd(M)$-equivariant rationalisation $\map_{\partial}(M,S^{d+2m})\to \map_{\partial}(M,S^{d+2m}_{\Q})$ under which the integral sphere action of $\Z^*$ on $\map_{\partial}(M,S^{d+2m})$ and the restriction of the sphere action of $\Q^*$ on $\map_{\partial}(M,S^{d+2m}_{\Q})$ to $\Z^*$ commute. The conclusion follows from Theorem \ref{thm:maintheoremforgeneralmanifolds}.
\end{proof}

Finally, we can get rid of the weights in Theorem \ref{thm:maintheoremforgeneralmanifolds}, if we choose a big enough $m$.
\begin{corollary}\label{thm:weighttheorembigm}
For $i,k\ge 1$ and $2m\ge max(2,i,kd)$, there is a $\MCG(M)$-equivariant isomorphism
\begin{equation*}
    H^i(C_k(M))\cong\widetilde{H}^{i+2mk}(\mapd(M;S^{d+2m}_{\Q})). 
\end{equation*}
\end{corollary}
\begin{proof}
Follows from the observation that $\widetilde{H}^i(C_k(M))$ vanishes for $i<0$ and $i>kd$, since $C_k(M)$ is a $kd$-dimensional manifold.
\end{proof}

\section{Based maps to spheres}\label{sec:mapsfromwedges}
In order to apply Theorem \ref{thm:maintheoremforgeneralmanifolds} to $\Sigma_{g,1}$, we need to compute $H^*(\mapd(\Sigma_{g,1},S^{2m}_{\Q}))$ with its $\Gamma_{g,1}\times\Q^*$-action for $m\ge 2$. As an intermediate step, we relax our maps to be merely basepoint-preserving and compute the algebra $H^*(\map_*(\Sigma_{g,1},S^{2m}_{\Q}))$ with its $\Gamma_{g,1}\times\Q^*$-action. We do this, by venturing into the land of free groups and computing $H^*(\map_*(\vee_nS^1,S^{2m}_{\Q}))$ for all $n\ge 1$, functorially with all based maps of wedges of circles (Theorem \ref{thm:mapsfromwedges}). The answer is given in terms of the Johnson action (Definition \ref{def:algerbaswithJohnsonaction}).
In the end, we present analogous results with $K(\Q,m)$ as codomain (Theorem \ref{thm:mapstoEM}) and, in the appendix, we free the basepoint to compute $H^*(\map(\vee_nS^1,S^{2m}_{\Q}))$ equivariantly (Theorem \ref{thm:freemaps}).  

\subsection{The content of free words}\label{sec:content}
Denote by $\Z^{*n}$ the free product of $n$ copies of $\Z$, i.e. the free group of rank $n$ with a choice of basis $\alpha_1,...,\alpha_n$, coming from each copy of $\Z$. The \textit{length} $l(w)$ of a word $w\in\Z^{*n}$ is measured with respect to this basis. 
Write $[-]:\Z^{*n}\to \Z^{n}$ for the abelianisation map, giving $\Z^n$ the standard basis $e_i:=[\alpha_i]$, $i=1,...,n$. Recall that the \textit{second exterior power} of a $\Z$-module $M$ is the $\Z$-module $\Lambda^2M=M\otimes_{\Z}M/\langle v\otimes v:v\in M\rangle$ in which the image of $v\otimes u$ is written as $v\wedge u$.
Write $\A=\{\alpha_1^{\pm},...,\alpha_n^{\pm}\}$ for the set of length-one words of $\Z^{*n}$.

\begin{definition}
The \textit{content} is the function \begin{equation*}
    c:\Z^{*n}\to \Lambda^2\Z^{n}
\end{equation*} defined by
\begin{equation}
    c(w)=\sum_{\mathclap{1\le i<j\le k}} \hspace{6pt} [w_i]\wedge [w_j]
\end{equation}
for each word $w=w_1w_2...w_k\in\Z^{*n}$ with $k\ge 0$ and $w_i\in \A$ for all $i$.
\label{def:content}
\end{definition}

\begin{proposition}
The content is well-defined. 
\end{proposition}
\begin{proof}
It suffices to check that if two consecutive letters are inverse, say $w_p^{-1}=w_{p+1}$, then $c(w)$ equals the content of the reduced word $w'=w_1...w_{p-1}w_{p+2}...w_k$  after cancelling $w_p,w_{p+1}$. The difference $c(w)-c(w')$ is by definition
\begin{small}
\begin{equation*}
    c(w)=\sum_{\mathclap{1\le i< p}}  \hspace{6pt} ([w_i]\wedge [w_p]+[w_i]\wedge [w_{p+1}])+[w_p]\wedge [w_{p+1}] +\sum_{\mathclap{p+1< j\le k}}  \hspace{6pt} ([w_p]\wedge [w_j]+[w_{p+1}]\wedge [w_j])
\end{equation*}
\end{small}
and $[w_{p+1}]=[w_p^{-1}]=-[w_p]$ so all three summands vanish.
\end{proof}
\begin{remark}
The content depends on the choice of basis of the free group $\Z^{*n}$. For example, if we wrote $c_{\alpha}$ for the content defined above and $c_{\beta}$ for the content defined using the basis $\beta_1:=\alpha_1\alpha_2$, $\beta_{i}=\alpha_i$ for $i=2,...,n$, then we would have $c_{\alpha}(\beta_1)=e_1\wedge e_2$ whereas $c_{\beta}(\beta_1)=0$.  
\label{rmk:contentnoncanonical}
\end{remark}
\begin{proposition}\label{prop:propertiesofcontent}
The content $c$ satisfies: 
      \begin{enumerate}
        \item $c(a)=0$ if $a=1$ or $a\in \A$;
        
        \item $c(a^{-1})=-c(a)$ for all $a\in \Z^{*n}$;
        
        \item $c(ab)=c(a)+c(b)+[a]\wedge[b]$ for all $a,b\in \Z^{*n}$;
        
        \item $c([a,b])=2[a]\wedge[b]$ where $[a,b]:=aba^{-1}b^{-1}$ is the commutator;
        
        \item $c([a,[b,c]])=0$ for all $a,b\in \Z^{*n}$ i.e. $c$ vanishes on double commutators. 
      \end{enumerate}
\end{proposition}
\begin{proof}
Properties (1),(2),(3) follow immediately from the definition by rearranging indices and using $[w^{-1}]=-[w]$ for all $w\in \Z^{*n}$. Property (4) follows by writing $[a,b]=aba^{-1}b^{-1}$, using property (3) three times to obtain 
\begin{align*}
    &c([a,b])=c(a)+c(b)+c(a^{-1})+c(b)\\
    &+[a]\wedge[b]+[a]\wedge[b^{-1}]+[a]\wedge[a^{-1}]+[b]\wedge[a^{-1}]+[b]\wedge[b^{-1}]+[a^{-1}]\wedge[b^{-1}]
\end{align*}
and cancelling using property (2). Finally, (5) follows from by applying (4) to get $c([a,[b,c])=2[a]\wedge[[b,c]]=0$ because commutators vanish in the abelianisation. 
\end{proof}

\begin{example}
For $n=2g$ and $\zeta=a_1a_2a_1^{-1}a_2^{-1}...a_{2g-1}a_{2g}a_{2g-1}^{-1}a_{2g}^{-1}$,
\begin{equation}
    c(\zeta)=2\big(e_1\wedge e_2+...+e_{2g-1}\wedge e_{2g}\big).
\end{equation}
The word $\zeta$ represents 
the boundary loop of $\Sigma_{g,1}$ in a standard basis of $\pi_1(\Sigma_{g,1})$. This computation will be used in the proof Lemma \ref{lem:sssofr}.
\label{ex:boundarywordcontent}
\end{example}

\subsection{The functor $H^*(\map_*(-,S^{2m}_{\Q}))$ on wedges of circles}\label{sec:xi}
Let $\hom(G,H)$ be the set of group homomorphisms between two groups $G,H$. 
\begin{definition}
    Let $\mathcal{F}$ be the category whose objects are the free groups $\Z^{*n}$, for $n\ge 0$, and morphisms are all group homomorphisms $\phi\in \hom(\Z^{*n}, \Z^{*m})$, for $n,m\ge 0$.
    \label{def:functoroffreegroups}
\end{definition}
We will be interested interested in functors from $\F$. Some examples are
\begin{itemize}
    \item the functor into graded commutative algebras $$\Z^n\to H^*(\map_*(\vee_nS^1,S^{2m}_{\Q})),$$ since $$\pi_0(\map_*(\vee_nS^1,\vee_{n'}S^1))=\hom(\Z^{*n},\Z^{*n'});$$
    
    \item the \textit{abelianisation} functor $[-]:\Z^{*}\to \Z^n$ into the category of abelian groups;
    
    \item writing $x_1,...,x_n\in \Q^n$ for the standard basis, the \textit{exterior algebra} functor $$\Z^n\to \Lambda[x_1,...,x_n]$$ 
    into the category of algebras, with the abelianisation action extended as algebras.
\end{itemize}

\begin{definition}[The function $\xi$]
For $n,n'\ge 1$, define  $$\xi:\hom(\Z^{*n},\Z^{*n'})\to \hom(\Q^{n}, \Lambda^2\Q^{n'})$$ by 
$$\xi(\phi)(e_i)=c(\phi(\alpha_i))$$
on the basis $\{e_i\}$ of $\Q^n$, and extended $\Q$-linearly. 
\label{def:xi}
\end{definition}
\begin{remark}
A priori, $\xi$ is merely a function, without extra properties.
\end{remark}

We henceforth fix $m\ge 1$ and $\mu=2m-1$.
\begin{definition} For $n\ge 0$, define the weighted graded commutative algebra
\begin{equation}
     A_{n}^{*,(*)}=A_{n,m}^{*,(*)}:=\Lambda[x_1,...,x_n]\otimes\Q[y_1,...,y_{2g}]
\end{equation}
with degrees/weights $|x_i|=(\mu,(1))$ and $|y_i|=(2\mu,(2))$. 

For each $\phi\in \hom(\Z^{*n},\Z^{*n'})$, define the map 
$J(\phi):A_{n}^{*,(*)}\to A_{n'}^{*,(*)}$
on the generators
\begin{center}
    $J(\phi)(x_i)=\{[\phi](e_i)\}_x$
\end{center}
and
\begin{center}
    $J(\phi)(y_i)=\{[\phi](e_i)\}_y+\{\xi(\phi)(e_i)\}_x,$
\end{center}
and extend it as an algebra map. Here, for an expression $p(e_i)$ in terms of the basis $e_i\in\Q^n$, we denote by $\{\alpha\}_x$, $\{\alpha\}_y$ the corresponding expression in terms of the bases $x_i$, $y_i$, respectively. 
\label{def:johnsonaction}
\end{definition}

Now, with $\alpha\in\pi_1(S^1)$ the generator given by the identity map, we obtain a canonical identification $\Z^{*n}=\pi_1(\vee_nS^1)$, so that $\alpha_i$ is the generating loop of the $i$-th circle. To simplify notation, for each $\phi\in\map_*(\vee_nS^1,\vee_{n'}S^1)$, we also denote by $\phi$ its homotopy class $$ \pi_0(\map_*(\vee_nS^1,\vee_{n'}S^1))=\hom(\Z^{*n},\Z^{*n'})$$ 
and by $\phi^*$ the induced map
$$\phi^*:H^*(\map_*(\vee_nS^1,S^{2m}_{\Q}))\to H^*(\map_*(\vee_{n'}S^1,S^{2m}_{\Q})).$$
Since $\phi^*$ depends only on the homotopy class of $\phi$, this will not cause confusion. 

\begin{theorem}
For every $n\ge 1$, the sphere action on $H^*(\map_*(\vee_nS^1,S^{2m}_{\Q}))$ is purely by weights and there are weighted graded algebra isomorphisms
$$H^*(\map_*(\vee_nS^1,S^{2m}_{\Q}))\cong A_{n,m}^{*,(*)}.$$

Under these isomorphisms, for any $n,n'\ge 1$ and map $\phi\in \map_*(\vee_nS^1,\vee_{n'}S^1)$ the induced map $\phi^*$ on $H^*(\map_*(-,S^{2m}_{\Q}))$ is given by $\phi^*= J(\phi)$.
\label{thm:mapsfromwedges}
\end{theorem}
\begin{proof}
We recall a few facts about the Hopf algebra $H^*(\Omega S^{2m}_{\Q})\cong H^*(\Omega S^{2m})$ (see e.g. \cite{Hatcher}). First, as graded commutative algebras,
$$H^*(\Omega S^{2m})\cong S[x,y]=\Lambda[x]\otimes \Q[y],$$ with degrees $|x|=\mu,|y|=2\mu$. We choose the generators $x$ and $y$ as follows. The Pontrjagin product of $\Omega S^{2m}$ makes $H_*(\Omega S^{2m};\Z)$ into the polynomial ring $\mathbb{Z}[t]$ with $|t|=\mu$, where $t$ is the Hurewicz image of the generator of $\pi_{2m-1}(\Omega S^{2m})$ adjoint to the identity map of $S^{2m}$. We define $x,y$ so that they evaluate to $1$ on $t,t^2$ respectively. The coproduct of 
$H^*(\Omega S^{2m})$ is induced by the loop-concatenation map $C:\Omega S^{2m}\times \Omega S^{2m}\to \Omega S^{2m}$ and given by $C^*(x)= x\otimes 1+1\otimes x$ and $C^*(y)= y\otimes 1+x\otimes x+1\otimes y$, and the involution is induced by the the loop-inversion map $I:\Omega S^{2m}\to \Omega S^{2m}$ and given by
$I^*(x)=-x$ and $I^*(y)=-y$.

Now, acting on $\Omega S^{2m}$ by a based map of $S^{2m}$ of degree $n$ multiplies $t$ by $n$ and thus $t^2$ by $n^2$. By dualising this, we conclude that the sphere action on $H^*(\Omega S^{2m}_{\Q})=S[x,y]$ is by weight $1$ on $x$ and weight $2$ on $y$. We identify $\map_*(S^1,S^{2m}_{\Q})=\Omega S^{2m}_{\Q}$ and hence
\begin{equation}
    \map_*(\vee_nS^1,S^{2m}_{\Q})=\big(\map_*(S^1,S^{2m}_{\Q})\big)^{\times n}=(\Omega S^{2m}_{\Q})^{\times n},
\label{eq:productofloopspaces}
\end{equation}
so by the K\"unneth theorem we obtain the $\Q^*$-equivariant identification
$$H^*(\map_*(\vee_nS^1,S^{2m}_{\Q}))=\Lambda[x_1,...,x_n]\otimes \Q[y_1,...,y_n]$$
with degrees $|x_i|=\mu$, $|y_i|=2\mu$ and weights $W(x_i)=1$, $W(y_i)=2$. The $x_i,y_i$ are the pullbacks of $x,y$ via the projection $p_i$ to the $i$-th factor in \eqref{eq:productofloopspaces}.

It remains to prove that the action of $\phi\in \map_*(\vee_nS^1,\vee_{n'}S^1)$ is given by $J(\phi)$ on the $x_i$ and $y_i$. Using the decomposition \eqref{eq:productofloopspaces}, it suffices to check only for $n=1$. In this case, the map $\phi:S^1\to \vee_nS^1$ induces $\Omega\phi:\big(\Omega S^{2m}_{\Q}\big)^{\times n}\to \Omega S^{2m}_{\Q}$ which is homotopic to a sequence of compositions of $C$ and $I$, as prescribed by the word $w=\phi(\alpha)\in \Z^{*n}$. More specifically, if $w=\alpha_{i_1}^{\varepsilon_1}\alpha_{i_2}^{\varepsilon_2}...a_{i_l}^{\varepsilon_l}$ with $\varepsilon_i\in\{\pm\}$, then
\begin{equation}
    \Omega\phi\simeq C\circ \big(\iota_{i_1}^{\varepsilon_1}\times \big(C\circ\big(\iota_{i_2}^{\varepsilon_2}\times\big(...\times \iota_{i_l}^{\varepsilon_l}\big)...\big)\big)\big)
    \label{eq:thecomposition}
\end{equation} 
where $\iota_i$ is the inclusion of the $i$-th factor in \eqref{eq:productofloopspaces}, $\iota_i^+=\iota_i$ and $\iota_i^-=\iota_i\circ I$. An induction on the length $l$ of $w$ concludes that the right hand side composition $C_{\phi}$ in \eqref{eq:thecomposition} maps $C_{\phi}^*(x)=\tau_x(w)$ and $C_{\phi}^*(y)=\tau_y(w)+\xi(\phi,w)$, that is, by $J(\phi)$, as desired. The inductive step uses the definition of the content \ref{def:content} and the observation that $(\iota_{i}^{\varepsilon})^*$ maps $x\mapsto \varepsilon x_i$ and $y\mapsto\varepsilon y_i$, so $\big(C\circ (f\times \iota_{i}^{\varepsilon})\big)^*$ maps $x\mapsto f^*(x)+\varepsilon x_i$ and $y\mapsto f^*(y)+\varepsilon y_i+f^*(x)\wedge \varepsilon x_i$.
\end{proof}
\begin{remark}
For a general space $X$, the spaces $H^*(\map_*(\vee_nS^1,X))$ together with their functorial free group structure can be determined using the Hopf-structure of $H^*(\Omega X)$, see \cite{kassabovHopfalgberas}. 
\end{remark}
\begin{corollary}
The mapping $\Z^{*n}\to A^{*,(*)}_n$ and $\phi\in \hom(\Z^{*n},\Z^{*n'})\to J(\phi)$ is a functor from $\F$ to the category of weighted graded commutative algebras.
\label{cor:Jphiisafunctoroverfreegroups}
\end{corollary}
\begin{proof}
The mapping from $\F$ to weighted graded commutative algebras with $\Z^{*n}\to H^*(\map_*(\vee_nS^1,S^{2m}_{\Q}))$ and $\phi\to \phi^*$ is a priori a functor and, in Theorem \ref{thm:mapsfromwedges}, we proved that this functor is naturally isomorphic to the desired mapping. 
\end{proof}

\subsection{The free Johnson representation and the Johnson action on algebras}\label{sec:JohnsonRep}
Now we fix $n\ge 1$ and recall some representation theory of $\Aut(\Z^{*n})$. The action of $\Aut(\Z^{*n})$ on the abelianisation $\Z^n=(\Z^{*n})_{ab}$ induces the short exact sequence of groups \begin{equation}
    1\to IA_n\to \Aut(\Z^{*n})\to \glnz\to 1,
\end{equation} 
where the \textit{free Torelli group} $IA_n$ is the kernel of this action. The rationalisation $H=(\Z^{*n})_{ab}\otimes \Q\simeq \Q^n$ is the \textit{standard linear representation}\footnote{We have used the same notation $H$ for the standard linear and standard symplectic representations, defined in the introduction. It should be clear from the context which one we mean.}  of $\Aut(\Z^{*n})$. Any $\Aut(\Z^{*n})$-representation that factors through $\glnz$ -- or, equivalently, with trivial $IA_n$-action -- is called \textit{linear}. The standard representation $H$ is linear, as is its dual $H^{\vee}$, and all subquotients of tensor powers of these two. 

\begin{proposition}[The free Johnson representation]
The linear maps $J(\phi)$, for each $\phi\in\Aut(\Z^{*n})$, make the vector space $$J:=H\oplus \Lambda^2H,$$ appearing in degree $2\mu$ of $A^*_{n}$,
into an $\Aut(\Z^{*n})$-representation which fits into a short exact sequence  
\begin{equation}
    \begin{tikzcd}[column sep=small]
        0\rar &\Lambda^2H\rar["i"]& J\rar & H\rar& 0.
    \end{tikzcd}
    \label{eq:definingextensionofJ}
\end{equation} 
of $\Aut(\Z^{*n})$-representations.
Furthermore, provided $n\ge 2$, the representation $J$ does not factor through $\glnz$, i.e. it is not linear, nor through $\Out(\Z^{*n})$.
\label{prop:JohnsonRep}
\end{proposition}
\begin{proof}
From Corollary \ref{cor:Jphiisafunctoroverfreegroups}, the linear maps $J(\phi)$, for $\phi\in\Aut(\Z^{*n})$, define an action on $J:=A^{2\mu,(2)}_{n}=H\oplus \Lambda^2H$. The action is linear on the subspace  $\Lambda^2H\subset J$ and linear on $J/\Lambda^2H\cong H$, giving the claimed short exact sequence.

That it does not factor through $\glnz$ nor $\Out(\Z^{*n})$ can be shown by exhibiting an element $\phi\in \Inn(\Z^{*n})\subset \Aut(\Z^{*n})$ which acts non-trivially in $J$, as the subgroup $\Inn(\Z^{*n})$ lies in the kernel of both maps from $\Aut(\Z^{*n})$ to $\Out(\Z^{*n})$ and $\GL_n(\mathbb{Z})$. We choose $$\phi(a_i)=a_1a_ia_1^{-1}$$ for $i=1,...,n$. Then we have $\phi\cdot y_i=y_i+2x_1x_i\neq y_i$ for $i\ge 2$.
\end{proof}

\begin{corollary}
The map $\xi:\Aut(\Z^{*n})\to \hom(H, \Lambda^2H)$, from Definition \ref{def:xi}, is a crossed homomorphism. 
\label{cor:xicrossedhom}
\end{corollary}
\begin{proof}
This follows immediately from Proposition \ref{prop:JohnsonRep} and the correspondence of extensions and crossed-homomorphisms.
\end{proof}

Recall the $\Aut(\Z^{*n})$-algebra $A_n^*=\Q[y_1,...,y_{n}]\otimes \Lambda[x_1,...,x_{n}]$ with the $J(\phi)$ action, for each $\phi\in \Aut(\Z^{*n})$. 
\begin{definition}[The Johnson action on the $y_i$]
    \label{def:algerbaswithJohnsonaction}
    Let $T$ be an $\Aut(\Z^{*n})$-algebra with a compatible $\Lambda[x_1,...,x_{n}]$-module structure. We specify a new $\Aut(\Z^{*n})$-algebra structure on the $\Q$-algebra $\Q[y_1,...,y_{n}]\otimes T$
    using the natural $\Aut(\Z^{*n})$-structure on the right hand side of $$\Q[y_1,...,y_{n}]\otimes T:=(\Q[y_1,...,y_{n}]\otimes \Lambda[x_1,...,x_{n}])\otimes_{\Lambda[x_1,...,x_{2g}]} T.$$
\end{definition}

\begin{theorem}
There is an isomorphism of weighted graded $\Aut(\Z^{*n})$-algebras
\begin{equation}
    H^*(\map_*(\vee_nS^1,S^{2m}_{\Q}))\cong \Q[y_1,...,y_{n}]\otimes \Lambda[x_1,...,x_{n}], 
\end{equation}
with degree/weight on the right given by $|x_i|=(\mu,(1))$ and $|y_i|=(2\mu,(2))$, and $\Aut(\Z^{*n})$-action the standard linear on the $x_i$ and the Johnson action on the $y_i$.
\label{thm:autfnmapsfromwedgestospheres}
\end{theorem}
\begin{proof}
Immediate from Theorem \ref{thm:mapsfromwedges}.
\end{proof}

\subsection{Maps from surfaces}\label{sec:mapsfromsurfaces}
We fix genus $g\ge 1$ and shift our attention to mapping spaces from $\Sigma_{g,1}$. We fix a basepoint on the boundary $p\in \partial\Sigma_{g,1}$ and a based deformation retraction $i:\vee_{2g}S^1\to \Sigma_{g,1}$ that identifies $\pi_1(\Sigma_{g,1})=\Z^{*2g}$ so that the basis $\alpha_1,...,\alpha_{2g}$ is a \textit{standard symplectic basis}. 

The action of $\Gamma_{g,1}$ on $\pi_1(\Sigma_{g,1})=\Z^{*{2g}}$ defines a homomorphism $\rho:\Gamma_{g,1}\to\Aut(\Z^{*2g})$, that is injective with image
\begin{equation}
    \im \rho= \{\phi\in \Aut(\Z^{*2g}):\phi(\zeta)=\zeta\}
\end{equation}
where $\zeta\in \pi_1(\Sigma_{g,1})$ is the word corresponding to the boundary loop, mentioned in Example \ref{ex:boundarywordcontent}. In other words, $\Gamma_{g,1}$ can be viewed as a subgroup of $\Aut(\Z^{*n})$. Consequently, we define the \textit{symplectic crossed homomorphism $\xi$}, the \textit{symplectic Johnson representation $J$} and \textit{$\Gamma_{g,1}$-algebras with Johnson action on the $y_i$} by restricting Definition \ref{def:xi}, Proposition \ref{prop:JohnsonRep} and Definition \ref{def:algerbaswithJohnsonaction} to $\Gamma_{g,1}$ via $\rho$.

\begin{theorem}
The sphere action on $H^*(\map_*(\Sigma_{g,1},S^{2m}_{\Q}))$ is purely by weights and there is an isomorphism of weighted graded $\Gamma_{g,1}$-algebras
\begin{equation}
    H^*(\map_*(\Sigma_{g,1},S^{2m}_{\Q}))\cong \Q[y_1,...,y_{2g}]\otimes \Lambda[x_1,...,x_{2g}]. 
    \label{eq:mapsfromsurfacestospheres}
\end{equation}
with degree/weight on the right given by $|x_i|=(\mu,(1))$ and $|y_i|=(2\mu,(2))$, and $\Gamma_{g,1}$-action the standard symplectic on the $x_i$ and the Johnson action on the $y_i$.
\label{thm:mapsfromsurfaces}
\end{theorem}
\begin{proof}
Immediate from Theorem \ref{thm:autfnmapsfromwedgestospheres} for $n=2g$, using the deformation retraction $i$ and the embedding of groups $\rho:\Gamma_{g,1}\to \Aut(\Z^{*2g})$.
\end{proof}

\subsection{Remarks on the Johnson representation}
For a group $G$, denote by $G^{(i)}$ its lower central series, namely $G^{(0)}=G$ and $G^{(i)}=[G,G^{(i-1)}]$ for $i\ge 1$. We thus obtain a short exact sequence of groups
\begin{equation}
    1\to G^{(1)}/G^{(2)}\to  G/G^{(2)}\to  G/G^{(1)}=G_{ab}\to 1.
    \label{eq:centralextensionofG}
\end{equation}
In the case $G=\Z^{*n}$, this sequence has a natural action of $\Aut(\Z^{*n})$. The outermost groups are $G_{ab}=\Z^n$ and $G^{(1)}/G^{(2)}\cong \Lambda^2\Z^n$, the latter via a map
\begin{align*}
    L^2:(\Z^{*n})^{(1)}=[\Z^{*n},\Z^{*n}]&\to \Lambda^2\Z^n\\
    [a,b]&\mapsto [a]\wedge[b]
\end{align*}
that has kernel precisely $(\Z^{*n})^{(1)}/(\Z^{*n})^{(2)}$
(see \cite{DarneAndreadakisIAn,LazardGroupesNilpotents54}: the Lie ring of $\Z^{*n}$ is the free Lie ring on $\Z^n$ in degree 1).
In this respect the content $c:\Z^{*n}\to \Lambda^2\Z^n$ is a non-canonical (in the sense of Remark \ref{rmk:contentnoncanonical}) extension of the map $2L^2$ to the whole of $\Z^{*n}$.

\subsection{Eilenberg-MacLane spaces and the Postnikov fibration of $S^{2m}_{\Q}$}\label{sec:oddsphereseilenbergmaclane}
\begin{definition}
    For $m,r\ge 1$, a based $\Q^*$-action on an Eilenberg MacLane space $(K(\Q,m),*)$ is called a \textit{degree $r$ action} if the induced action on $\pi_m(K(\Q,m))=\Q$ is by multiplication by $q^r$ for every $q\in \Q^*$. We write $K(\Q(r),m)$ to remember such an action.
    \label{def:poweredEMspaces}
\end{definition}
\begin{example}
A sphere $S^{2m+1}_{\Q}$, $m\ge 0$, with a sphere action is a $K(\Q(1),2m+1)$.
\end{example}

We now pick a topological model for the Postnikov fibration of $S^{2m}_{\Q}$, for $m\ge 1$, with a commuting sphere action. We use Chapter 17 from \cite{FelixYves2001RHTb} to sketch such a construction. A Sullivan model for $S^{2m}$ is $(\Lambda^*(u[2m],v[4m-1]); d(v)=u^2)$. It fits in the short exact sequence of augmented Sullivan algebras
\begin{equation}
    \small
    \begin{tikzcd}[row sep=small, column sep=small]
    (\Lambda^*(v[4m-1]), 0)\arrow[dr] & (\Lambda^*(u[2m],v[4m-1]), d(v)=u^2)\lar\dar & (\Lambda^*(u[2m]), 0)\arrow[dl]\lar\\
    & (\Q,0) & 
    \end{tikzcd}.
    \label{eq:SullivanAlgebras}            
\end{equation}
and the group $\Q^{*}$ acts on the diagram by algebra automorphisms via $q\cdot u=qu$ and $q\cdot v=q^2v$. Applying  the \textit{spacial realization} functor from \textit{dgca}s to \textit{topological spaces} translates diagram \eqref{eq:SullivanAlgebras} to fibration \eqref{eq:fibrationrationalsphere} to get a based fibration sequence
    \begin{equation}
    \begin{tikzcd}
     (K(\Q(2),4m-1),*)\rar & (S^{2m}_{\Q},*)\rar["u"] & (K(\Q(1),2m),*)
    \end{tikzcd}
    \label{eq:fibrationrationalsphere}
    \end{equation}
with a commuting sphere action in the middle and degree actions on the base and fibre.

We compute the $\Gamma_{g,1}\times \Q^*$-algebra $H^*(\map_*(\Sigma_{g,1},K(\Q(r),m)))$ for $m,r\ge 1$, and how the functor $H^*(\map_*(\Sigma_{g,1},-))$ applies to the Postnikov fibration \eqref{eq:fibrationrationalsphere}.

\begin{theorem}
For all $m\ge 2$, the action of $\Q^*$ on $H^*(\map_*(\Sigma_{g,1},K(\Q(r),m)))$ is purely by weights and there are isomorphisms of weighted graded commutative $\Gamma_{g,1}$-algebras
\begin{equation}
    H^*(\map_*(\Sigma_{g,1},K(\Q(r),m));\Q)\cong S[e_1,...,e_{2g}],
\end{equation}
with degrees/weights $|e_i|=(m-1,(r))$ and the symplectic action on the $e_i$.

Furthermore, under these isomorphisms and isomorphism \eqref{eq:mapsfromsurfacestospheres}, the functor $H^*(\map_*(\Sigma_{g,1},-))$ applied to fibration~\eqref{eq:fibrationrationalsphere} 
$$K(\Q(1),2m)\longleftarrow S^{2m}_{\Q}\longleftarrow K(\Q(2),4m-1))$$
is given by
\begin{equation*}
\Lambda[x_1,...,x_{2g}]\to
\Q[y_1,...,y_{2g}]\otimes \Lambda[x_1,...,x_{2g}] \to \Q[y_1,...,y_{2g}],
\end{equation*}
where the first map is the inclusion and the second is the projection.
\label{thm:mapstoEM}
\end{theorem}
\begin{proof}
For the first part, the argument is identical to what we have done for spheres: identify $H^*(\Omega K(\Z,m))=S[x[m-1,(1)]]$ with Hopf-algebra structure defined by $x\mapsto x\otimes 1+1\otimes x$; an identical argument to Proposition \ref{thm:mapsfromwedges}, shows that $H^*(\map_*(\vee_n S^1,K(\Q,m)))=S[e_1,...,e_{2g}]$ with no need to define any quantity analogous to the content; transfer it to surfaces to obtain the result. Of course that means that changing to $K(\Q(r),m)$ we must multiply all weights by $r$.  

For the second part, we observe that the spherical fibration \eqref{eq:fibrationrationalsphere} gives maps on rational cohomology $H^*(\Omega K(\Q(1),2m))\to H^*(\Omega S^{2m}_{\Q})\to H^*(\Omega K(\Q(2),4m-1))$ given by
$$\Lambda[x[2m-1,(1)]]\to \Lambda[x[2m-1,(1)]]\otimes\Q[y[4m-2,(2)]]\to \Q[y[4m-2,(2)]].$$
From this the second part follows.
\end{proof}

\subsection{Appendix: Free maps to spheres}\label{sec:freemaps}
Here we compute the weighted $\Aut(\Z^{*n})\times \Q^*$-algebra structure of $H^*(\map(\vee_nS^1,S^{2m}_{\Q}))$ for free (i.e. unbased) maps. This result is not needed for the remainder of this paper, except briefly in Section \ref{sec:closedsurfaces}, as it relates to computing the $\MCG(\Sigma_g)$-action on $H^*(C_n(\Sigma_g))$ for closed surfaces. We fix $n,m\ge 1$. Recall algebra the $\Aut(\Z^{*n})$-algebra $A_n^*=\Q[y_1,...,y_{n}]\otimes \Lambda[x_1,...,x_{n}]$

\begin{definition}
The \textit{Koszul differential} $d_K:A_n^*\to A_n^*$ is given by $d_K(x_i)=0$ and $d_K(y_i)=x_i$ for $i=1,...,n$, and extended as a derivation of degree $(-\mu, (0))$ via the Leibniz rule.
\end{definition}
\begin{proposition}
The derivation $d_K$ is an $\Aut(\Z^{*n})$-equivariant differential with homology $H_*(A_n,d_K)\cong \Q$ generated by the class of the unit $1\in K$.
\end{proposition}
\begin{proof}
It is easy to check equivariance and that it is a differential. The homology follows from writing the generators as $y_i$ and $d_Ky_i=x_i$, and observing that this is the homology of differential forms on an $n$-simplex.
\end{proof}

The kernel $\ker d_K$ is a weighted graded $\Aut(\Z^{*n})$-subalgebra of $A_n^{*,(*)}$ and the cokernel $\coker d_K$ inherits a weighted graded module structure over $\ker d_K$ from the multiplication in $A_n^{*,(*)}$. Therefore we can speak of the square-zero extension
\begin{equation}
    \ker d_K\ltimes \coker d_K[2m,(1)]
\end{equation}
where $\coker d_K[\mu+1,(1)]$ is $\coker d_K$ shifted up by $(\mu+1,(1))$ degrees. The square-zero extension is naturally a weighted $\Aut(\Z^{*n})$-algebra.

\begin{theorem}\label{thm:freemaps}
There is an isomorphism of weighted $\Aut(\Z^{*n})$-algebras 
\begin{equation}
    H^*(\map_*(\vee_nS^1,S^{2m}_{\Q}))\cong \ker d_K\ltimes \coker d_K[\mu+1,(1)].
\end{equation}
\end{theorem}
\begin{proof}
The Serre spectral sequence associated with the evaluation fibration
\begin{equation}
\begin{tikzcd}
 \map_*(\vee_nS^1,S^{2m}_{\Q})\rar&\map(\vee_nS^1,S^{2m}_{\Q})\rar["\ev"]& S^{2m}_{\Q},
\end{tikzcd}
\label{eq:fibrationunbasedmaps}
\end{equation}
where $\ev$ evaluates maps at the wedge point, has the form
\begin{equation}
    E_2^{p,q}=H^p(S^{2m}_{\Q})\otimes H^q(\map_*(\vee_nS^1,S^{2m}_{\Q}))\implies H^{p+q}(\map(\vee_nS^1,S^{2m}_{\Q})).
    \label{eq:SSSunbasedmaps}
\end{equation}
Its only two non-trivial columns are $p=0$ and $2m=\mu+1$, and so the only possibly non-trivial differentials are
$$d_{2m}:\Q\{1\}\otimes S^q(H[\mu]\oplus H[2\mu])\to \Q\{u\}\otimes S^{q-\mu}(H[\mu]\oplus H[2\mu])$$
where $u\in H^{2m}(S^{2m}_{\Q})$ is the orientation class.

We claim that $d_{2m}$ is given by $d_{2m}(1\otimes z)=2u\otimes d_K(z)$. To prove this, since $d_{2m}$ is also a derivation, it suffices to check on $z=x_i,y_i$ for $i=1,...,n$. By using the retractions of $\vee_nS^1$ to its wedge circles, we need only check for $n=1$. But in that case the fibration is precisely the free loopspace fibration of $S^{2m}_{\Q}$. This is studied in \cite{Seeliger2011freeloopspace} for $S^{2m}$ and indeed has $d_{2m}(1\otimes x)=0$ and $d_{2m}(u\otimes y)=2u\otimes x$. It immediately follows that the bigraded $\Aut(\Z^{*n})$-algebra $E^{*,*}_{\infty}$ is given by $$1\otimes \ker d_K\ltimes u\otimes\coker d_K$$
where $(\ker d_K)^i$ and $(\coker d_K)^i$ are in degrees $(i,0)$, and $|u|=(0,2m)$. 

Finally, the $\Q^*$-action on the $E_2$-page is by weight $k$ on $H^{k\mu}(\map_*(\vee_nS^1,S^{2m}_{\Q}))$ and by weight $1$ on $u$. So it is by weight $k$ on $E_{\infty}^{0,k\mu}$ and weight $k+1$ on $E_{\infty}^{2m,kq'}$, and these are the only non-trivial entries of $E_{\infty}$. If $m>1$, then on every diagonal $p+q=d$ there is at most one non-trivial entry, and we conclude that $H^d(\map(\vee_nS^1,S^{2m}_{\Q}))$ is equal to that entry $\Aut(\Z^{*n})\times\Q^*$-equivariantly. We thus conclude that cup-product, the $\Aut(\Z^{*n})$-action and the weights on $H^{*}(\map(\vee_nS^1,S^{2m}_{\Q}))$ are as desired. In the case $m=1$, we have $\mu=1$ and $H^{k+2}(\map(\vee_nS^1,S^{2m}_{\Q}))$ is an extension of $E_{\infty}^{0,k+2}$ by $E_{\infty}^{2,k}$ which a priori could be non-trivial. However, the two have different weights, namely $k+2$ and $k+1$ and the weightspaces thus provide an $\Aut(\Z^{*n})\times\Q^*$-equivariant splitting that is compatible with the cup-product. We conclude again the desired weighted $\Aut(\Z^{*n})$-algebra structure.
\end{proof}

\subsection{Relations with Higher Hochschild homology}
The $\Aut(\Z^{*n})$ representations in Theorems \ref{thm:autfnmapsfromwedgestospheres} and \ref{thm:freemaps} seem identical to the ones studied in \cite{turchin2016hochschildpirashvili,vespapowell2019higher}. The missing link is the relation of Higher Hochschild homology with homology of mapping spaces established in \cite{PirashviliHigherHochschild}. The benefit of working with cohomology of spaces is the algebra structure provided by the cup product. This allows us to determine the $\Aut(\Z^{*n})$-action on the algebra by describing it only on a generating set.

\section{Maps with fixed boundary behaviour}\label{sec:mapsrelboundary}
In this section, we compute the $\Gamma_{g,1}\times\Q^*$-algebra  $H^*(\mapd(\Sigma_{g,1},S^{2m}_{\Q}))$, for $g\ge 0$ and $m\ge 2$. Again, we set $\mu=2m-1$.

\begin{definition}\label{def:bigradedR}
    For $m\ge 1$, let $$R^{*,*,(*)}=\Q[y_1,...,y_{2g},w]\otimes \Lambda[x_1,...,x_{2g},v]$$ 
    be the weighted bigraded differential $\Gamma_{g,1}$-algebra with  
    \begin{itemize}
    \item bidegrees/weights $|x_i|=(\mu,0,(1))$, $|y_i|=(2\mu,0,(2))$,
    $|w|=(0,\mu-1,(1))$ and $|v|=(0,2\mu-1,(2))$,
    
    \item differential $d_R$ given by 
    \begin{equation*}
    d_R(v)=2\omega:=2(x_1x_2+...+x_{2g-1}x_{2g}),
    \end{equation*}
    vanishing on the $x_i$, $y_i$ and $w$, and extended using the Leibniz rule, and  
    
    \item the $\Gamma_{g,1}$ action trivial on $w$ and $v$, symplectic on the $x_i$ and the Johnson on the $y_i$.
    \end{itemize}
\end{definition} 

\begin{definition}[Totalisation]
For a bigraded algebra $B^{*,*}$, its \textit{totalisation} $\Tot(B)^*$ is the graded algebra $\Tot(B)^k=\oplus_{p+q=k}B^{p,q}$, with product inherited from $B^{*,*}$.
\end{definition}
\begin{remark}
If $B^{*,*}$ is weighted, then $\Tot(B)^{*}$ is also weighted. If $(B^{*,*},d_B)$ is a differential algebra, then $\Tot(B)^*$ inherits a differential $d_B$ and $H^*(\Tot(B),d_B)=\Tot(H(B,d_B))^*$. 
\end{remark}

\begin{definition}
For $m\ge 1$, define the weighted graded differential $\Gamma_{g,1}$-algebra
\begin{equation}
    R^{*,(*)}_m:=\Tot(R^{*,*,(*)}_m).
\end{equation} 
We call $R_m^{*,*}$ the \textit{bigraded version} of $R_m^{*}$.
\label{def:algebraR}
\end{definition}

\begin{theorem}\label{thm:maintheoremwithoutmap}
Let $g\ge 0$ and $m\ge 2$.
The sphere action on $H^*(\map_{\partial}(\Sigma_{g,1},S^{2m}_{\Q}))$ is purely by weights and there is a weighted $\Gamma_{g,1}$-equivariant isomorphism 
\begin{equation}
    \begin{tikzcd}
          H^*(\map_{\partial}(\Sigma_{g,1},S^{2m+2}_{\Q}))\cong H^*(R_m,d_R).
    \end{tikzcd}
    \label{eq:maintheoremwithoutmap}
\end{equation}
\end{theorem}

\begin{remark}\label{lem:allalgebrasSarethesame}
The algebras $R^{*,(*)}_m$ for different $m\ge 1$ are equal as weighted ungraded differential $\Gamma_{g,1}$-algebras. They only differ in the grading, with $R^{i,(n)}_1\cong R^{i+(\mu-1)n,(n)}_m$. These relations on degrees/weights can be checked on the generators and extended to the whole algebra. The same statements thus hold for the homologies $H^{*,(*)}(R_m,d_R)$.
\end{remark}

\begin{remark}\label{rmk:alternativedescriptionofR}
An alternative description of the algebra $R_m^{*,(*)}$ is \begin{equation}
    R_m^{*,(*)}=S\Big(J[2\mu,(2)]\oplus H[\mu,(1)]\oplus \Q[\mu-1,(1)]\oplus \Q[2\mu-1,(2)]\Big)/I
\end{equation}
where
\begin{equation}
    I=\Big(z-i(z):z\in \Lambda^2\big(H[\mu,0,(1)]\big)\Big)
\end{equation}
and the inclusion $i:\Lambda^2\big(H[\mu,0,(1)]\big)\to J[2\mu,0,(2)]$ from \eqref{eq:definingextensionofJ}. We will not use this description in this section.
\end{remark}

For the rest of the section, we fix $m\ge 2$ and drop the subscript from $R^{*,(*)}_m$. 

\subsection{The topological input}\label{sec:thefibrations}
The topological input for the proof of Theorem \ref{thm:maintheoremwithoutmap} comes from the following $\Q^*$-equivariant commutative diagram.
\begin{equation}
    \begin{tikzcd}[column sep=tiny]
        \Omega^2S^{2m}_{\Q}\rar["j"]\dar[equal]&
        F_{\chi\circ r}\rar["\bar{r}"]\dar[hook,"i_{\chi\circ r}"]
        \arrow[dr, phantom, "\scalebox{1.5}{$\lrcorner$}" , very near start]
        & \map_*\big(\Sigma_{g,1}',K(\Q(2),4m-1)\big)
        \dar[hook,"i_{\chi}"]\\
        \Omega^2S^{2m}_{\Q}\rar&
        \map_{\partial}\big(\Sigma_{g,1},S^{2m}_{\Q}\big)\rar["r"]\dar["\chi\circ r"]& \map_*\big(\Sigma_{g,1}',S^{2m}_{\Q}\big)\dar["\chi"]\\
        &\map_*\big(\Sigma_{g,1}',K(\Q(1),2m)\big)\rar[equal] &
        \map_*\big(\Sigma_{g,1}',K(\Q(1),2m)\big)
    \end{tikzcd}
    \label{eq:bigdiagram}
\end{equation}
We will now define the spaces and maps appearing here, explain why the right two columns and top two rows are fibration sequences, and compute the Serre spectral sequences of the two fibrations with total space  $\map_{\partial}\big(\Sigma_{g,1},S^{2m}_{\Q}\big)$.

\subsubsection{Middle row: fibration $r$} We fix a smooth self-embedding $\rho:\Sigma_{g,1}\to \Sigma_{g,1}$ isotopic to the identity and supported in a collar neighbourhood of the boundary $\partial\Sigma_{g,1}$ that fixes the basepoint $p\in \partial\Sigma_{g,1}$ and pushes the rest of the boundary slightly inwards. Write $\Sigma_{g,1}':=\im \rho\subset \Sigma_{g,1}$. The map $r$ is defined as the restriction $\map_{\partial}(\Sigma_{g,1},S^{2m})\to \map_*(\Sigma_{g,1}',S^{2m})$. It is a fibration by virtue of being a restriction. The complement $\Sigma_{g,1}-\mathring{\Sigma}_{g,1}'$ is a closed disc
$D^2$ with two points on the boundary identified, so the fibre of $r$ is $\map_\partial(D^2,S^{2m}_{\Q})=\Omega^2S^{2m}_{\Q}$. We have the fibration sequence 
\begin{equation}
\begin{tikzcd}
    \Omega^2S^{2m}_{\Q}\rar& \map_{\partial}(\Sigma_{g,1},S^{2m}_{\Q})\rar["r"]& \map_*(\Sigma_{g,1}',S^{2m}_{\Q}).
    \end{tikzcd}
    \label{eq:fibrationR}
\end{equation}
appearing in the middle row of diagram \eqref{eq:bigdiagram}. 

\subsubsection{Right column: fibration $\chi$}
We recall Postnikov fibration \eqref{eq:fibrationrationalsphere} of $S^{2m}_{\Q}$ and apply on it the fibration preserving functor $\map_*(\Sigma_{g,1}',-)$ to obtain the fibration sequence
\begin{equation}
    \begin{tikzcd}[column sep=small]
    \map_*(\Sigma_{g,1}',K(\Q(2),4m-1))\rar["i_{\chi}"]& \map_{\partial}(\Sigma_{g,1}',S^{2m}_{\Q})\rar["\chi"]& \map_*(\Sigma_{g,1}',K(\Q(1),2m)).
    \end{tikzcd}
    \label{eq:fibrationchi}
\end{equation}
appearing in the right column of diagram \eqref{eq:bigdiagram}.

\subsubsection{The middle column: fibration $\chi\circ r$}
Composing two fibrations produces another fibration, thus we obtain the fibration sequence
\begin{equation}
    \begin{tikzcd}
    F_{\chi\circ r}\rar["i_{\chi\circ r}"]&
    \map_{\partial}(\Sigma_{g,1},S^{2m}_{\Q})\rar["\chi\circ r"]& \map_*(\Sigma_{g,1}',K(\Q(1),2m)).
    \end{tikzcd}
    \label{eq:fibrationchiR}
\end{equation}
appearing in the middle column of diagram \eqref{eq:bigdiagram}. 
By virtue of both maps $\chi$ and $r$ being fibrations, the fibre $F_{\chi\circ r}$ is the topological pullback in the top-right square of diagram \eqref{eq:bigdiagram}
\begin{equation}
    \begin{tikzcd}[column sep=tiny]
        F_{\chi\circ r}\rar["\bar{r}"]\dar[hook,"i_{\chi\circ r}"]
        \arrow[dr, phantom, "\scalebox{1.5}{$\lrcorner$}" , very near start]
        & \map_*\big(\Sigma_{g,1}',K(\Q(2),4m-1)\big)
        \dar[hook,"i_{\chi}"]\\
        \map_{\partial}\big(\Sigma_{g,1},S^{2m}_{\Q}\big)\rar["r"]& \map_*\big(\Sigma_{g,1}',S^{2m}_{\Q}\big)
    \end{tikzcd}
    \label{eq:pullbackdiagram}
\end{equation}
Consequently, the map $\bar{r}$ is the pullback of fibration $r$ under $i_{\chi}$, so it is a fibration as well, and we write $j:\Omega^2S^{2m}_{\Q}\to F_{\chi\circ r}$ for the inclusion of its fibre.

\subsubsection{The $\Gamma_{g,1}$ action}
The map $r$ and all other maps of diagram \eqref{eq:bigdiagram} are $\Diffd(\Sigma_{g,1}')$-equivariant. Since the inclusion $\rho:\Sigma_{g,1}\to \Sigma_{g,1}$ is an isotopy equivalence, the mapping class group $\Gamma_{g,1}$ will henceforth be viewed as $\MCG(\Sigma_{g,1}')$ instead of $\MCG(\Sigma_{g,1})$ without affecting the validity of our computations.

\subsection{The Serre spectral sequences}
We analyse the Serre spectral sequences of all fibrations appearing in diagram \eqref{eq:bigdiagram}. All properties of the Serre spectral sequence are taken from \cite{HatcherSSS}. For a fibration $f$ denote by $E^{p,q}_{r}(f)$ the entries of its Serre spectral sequence, with $p$ called the \textit{column}, $q$ the \textit{row} and $r$ the \textit{page}. 
Due to the commuting actions of $\Gamma_{g,1}$ and $\Q^*$, all $E^{*,*}_r(f)$ will be weighted bigraded $\Gamma_{g,1}$-algebras, with weight preserving $\Gamma_{g,1}$-equivariant differentials.
Observe that the weighted cohomology $\Gamma_{g,1}$-algebras of all of the spaces in the bottom row and right column of diagram \eqref{eq:bigdiagram} are known from Section \ref{sec:mapsfromwedges}. All base spaces of these fibrations are manifestly simply connected, as $m\ge 2$, so we need not worry about local coefficients. 

\begin{lemma}
    The Serre spectral sequence of fibration $r$ has only non-trivial differential $d_{4m-3}$ and there is an isomorphism  
    \begin{equation}
        (E_2^{*,*}(r),d_{4m-3})\cong (R^{*,*},d_R).
    \end{equation}
    of weighted differential $\Gamma_{g,1}$ algebras.
    \label{lem:sssofr}
\end{lemma}
\begin{proof}
This is a $\Gamma_{g,1}$-equivariant improvement of Theorem A in \cite{bodigheimer88rationalcohomologyofsurfaces}. The fibration $r$ is the pullback of the path-fibration of $\Omega S^{2m}_{\Q}$ under the restriction map $b:\map_*(\Sigma_{g,1}',S^{2m})\to \map_*(\partial\Sigma_{g,1}',S^{2m})=\Omega S^{2m}$. Observe $\Omega^2 S^{2m}_{\Q}=\map_{\partial}(D^2,S^{2m}_{\Q})$.
\begin{equation}
    \begin{tikzcd}
    \Omega^2S^{2m}_{\Q}\rar\dar[equal] & \map_{\partial}(\Sigma_{g,1},S^{2m}_{\Q})
    \rar["r"]\dar["\bar{b}"]
    \arrow[dr, phantom, "\scalebox{1.5}{$\lrcorner$}" , very near start] & 
    \map_*(\Sigma_{g,1}',S^{2m}_{\Q})\dar["b"] \\
    \Omega^2 S^{2m}_{\Q}\rar& \map_*(D^2,S^{2m}_{\Q})\rar["r_{\partial D^2}"] & 
    \Omega S^{2m}_{\Q}
    \end{tikzcd}
    \label{eq:pathspaceomegasphere}
\end{equation}

The spectral sequence of fibration $r_{\partial D^2}$ has
\begin{equation}
    E_2^{*,*}(r_{\partial D^2})=S\Big[w[0,\mu-1,(1)], v[0,2\mu-1,(2)], x[\mu,0,(1)], y[2\mu,0,(2)]\Big]
    \label{eq:SSSomega2spheres}
\end{equation}
and only non-trivial differentials $d_{\mu}(w)=x$ and $d_{2\mu}(v)=y$.
Therefore the fibre of $r$ has weighted cohomology algebra $H^*(\Omega^2S^{2m}_{\Q})=S\big[w[0,\mu-1,(1)], v[0,2\mu-1,(2)]\big]$ with trivial $\Gamma_{g,1}$ action, as $\Diffd(\Sigma_{g,1}')$ acts trivially on the fibre.
Using the weighted $\Gamma_{g,1}$-algebra structure of $H^*(\map_*(\Sigma_{g,1}',S^{2m}_{\Q}))$ from Theorem \ref{thm:mapsfromsurfaces}, we get the desired isomorphism $E_2^{*,*}(r)\cong R^{*,*}$.

We use the map $b$ to compute the differentials. Up to homotopy $b$ is the map $\map_*(\vee_{2g}S^1,S^{2m})\to \map_*(S^1,S^{2m})$ corresponding to the homomorphism $\Z\to \Z^{*2g}$ defined by $1\mapsto \zeta$ where \begin{equation}
    \zeta=a_1a_2a_1^{-1}a_2^{-1}...a_{2g-1}a_{2g}a_{2g-1}^{-1}a_{2g}^{-1}\in\Z^{*2g}.
\end{equation} 
From Proposition \ref{thm:mapsfromwedges} and Example \ref{ex:boundarywordcontent}, we get that $b^*:\Q[y]\otimes \Lambda[x]\to \Q[y_1,...,y_{2g}]\otimes \Lambda[x_1,...,x_{2g}]$ maps $b^*(x)=0$ and $b^*(y)=2\omega$. Therefore, by comparing with the differentials of \eqref{eq:SSSomega2spheres}, we obtain that $d_{m-1}(w)=0$ and $d_{2m-2}(v)=2\omega$ in $E^{*,*}_*(r)$.
\end{proof}

The fibrations $r$ and $\chi\circ r$ are compared by the map of fibrations $(\id_M,\chi)$ as shown in diagram
\begin{equation}
    \begin{tikzcd}
    \Omega^2S^{2m}_{\Q}\rar[hook,"i"]\dar["j"] &
    \map_{\partial}(\Sigma_{g,1},S^{2m}_{\Q})\rar["r"]\dar[equal]
    &\map_{*}(\Sigma_{g,1},S^{2m}_{\Q})\dar["\chi"]\\
    F_{\chi\circ r}\rar["i_{\chi\circ r}"]&
    \map_{\partial}(\Sigma_{g,1},S^{2m}_{\Q})\rar["\chi\circ r"] &
    \map_*(\Sigma_{g,1}',K(\Q,2m)).
    \end{tikzcd}
\label{eq:diagramcomparison of fibrations}    
\end{equation}

\begin{definition}[The algebra $Q$]\label{def:algebraQ}
Let
    \begin{equation}
    Q^{*,*,(*)}_m:=\Q[y_1,...,y_{2g},w]\otimes \Lambda[x_1,...,x_{2g},v]
\end{equation}
be the weighted graded differential $\Gamma_{g,1}$-algerba with 
\begin{itemize}
    \item degrees/weights $|x_i|=(\mu,0,(1))$, $|y_i|=(0,2\mu,(2))$,
    $|w|=(0,\mu-1,(1))$ and $|v|=(0,2\mu-1,(2))$,
    
    \item differential $d_Q$ given by 
\begin{equation}
    d_Q(v)=2\omega=2(x_1x_2+...+x_{2g-1}x_{2g}),
\end{equation}
    vanishing on the $x_i$, $y_i$ and $w$, and extended using the Leibniz rule, and  
    
    \item the $\Gamma_{g,1}$-action trivial on $w$ and $v$, symplectic both on the $x_i$ and the $y_i$.
\end{itemize}
\end{definition}
\begin{remark}\label{rmk:comparisonRQ}
As a $\Q$-algebra, $\Tot(Q,d_Q)^*=\Tot(R,d_R)^*$, but not as $\Gamma_{g,1}$-algebras: they differ in the $\Gamma_{g,1}$-action on the $y_i$. The $\Q$-algebra $(Q^{*,*},d_Q)$ and the bigraded version $(R^{*,*},d_R)$ themselves only differ in the bidegrees of the $y_i$, with $|y_i|_R=(2\mu,0)$ but $|y_i|_Q=(0,2\mu)$. See Figure \ref{fig:algerbasRQ} for a schematic.
\end{remark}

\begin{lemma}
     The Serre spectral sequence of $\chi\circ r$ has only non-trivial differential $d_{4m-3}$ and there is an isomorphism 
    \begin{equation}
        (E_2^{*,*},d_{4m-3})\cong (Q^{*,*},d_Q).
    \end{equation}
    of weighted differential $\Gamma_{g,1}$-algebras.
    
    Furthermore, under this isomorphism and the one from Lemma \ref{lem:sssofr}, the map of fibrations $(\id_M,\chi)$ induces the map of spectral sequences 
    $E^{*,*}_{2}(\chi\circ r)\to E^{*,*}_{2}(r)$ given by the map $$\phi:R^{*,*}\to  Q^{*,*}$$ that maps the $w,v,x_i$ to themselves and vanishes on the $y_i$.
    \label{lem:sssofchir}
\end{lemma}

\begin{figure}
    \centering
    \includegraphics[width=0.7\textwidth]{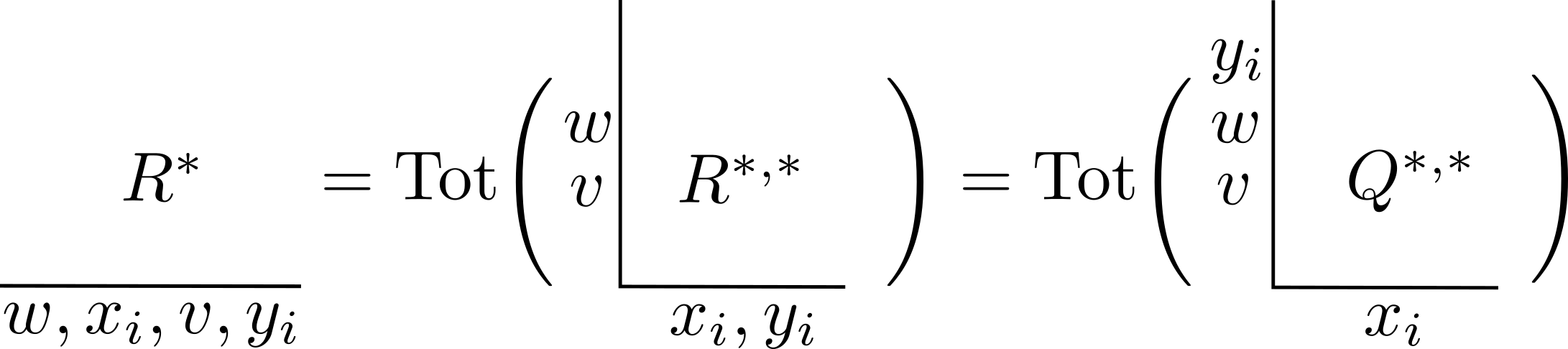}
    \caption{\footnotesize The algebras $R^{*,*}$ (Definition \ref{def:bigradedR}) and $Q^{*,*}$ (Definition \ref{def:algebraQ}) as totalisations of $R^*$. The first equality is $\Gamma_{g,1}$-equivariant, the second is not.}
    \label{fig:algerbasRQ}
\end{figure}
\begin{proof}
We compute $H^*(F_{\chi\circ r})$ by running the Serre spectral sequence of $\bar{r}$ from diagram \eqref{eq:bigdiagram}. Importing the cohomology of the fibre from Lemma \ref{lem:sssofr} and the basis from Theorem \ref{thm:mapstoEM}, we deduce that
\begin{equation}
    E_2^{*,*}(\bar{r})=S\Big[y_1[2\mu,0,(2)],...,y_{2g}[2\mu,0,(2)],w[0,\mu-1,(1)], v[0,2\mu-1,(2)]\Big]. 
    \label{eq:sssofbarr}
\end{equation}
Using the map of fibrations $(i_{\chi\circ r},i_{\chi})$ to compare with $r$ and Lemma \ref{lem:sssofr}, we conclude that all differentials in $E_*^{*,*}(\bar{r})$ are trivial, so $E_{\infty}^{*,*}(\bar{r})=E_{2}^{*,*}(\bar{r})$ is a free graded commutative algebra, with the generators $y_i,w,v$ separated by weight and total degree. Therefore
\begin{equation}\label{eq:homologyFchir}
    H^*(F_{\chi\circ r})=S\Big[y_1[2\mu,(2)],...,y_{2g}[2\mu,(2)],w[\mu-1,(1)], v[2\mu-1,(2)]\Big]
\end{equation}
with the symplectic $\Gamma_{g,1}$-action on the $y_i$ and the trivial on $v,w$.

From \eqref{eq:homologyFchir} and Theorem \ref{thm:mapstoEM}, the Serre spectral sequence of $\chi\circ r$ can be immediately deduced to have $E_2^{*,*}(\chi\circ r)\cong Q^{*,*}$ as weighted $\Gamma_{g,1}$-algebras.

Now, the induced map on $E_2^{*,*}(\chi\circ r)\to E_2^{*,*}(r)$ from the map of fibrations $(\id_M,\chi)$ in \eqref{eq:diagramcomparison of fibrations} is precisely $\chi^*\otimes j^*$. We compute it under the identifications $E^{*,*}_2(r)=R^{*,*}$ and $E^{*,*}_2(\chi \circ r)=Q^{*,*}$. 
Using the last part of Theorem \ref{thm:mapstoEM}, the map $\chi^*:H^*(\map_*\big(\Sigma_{g,1}',K(\Q,2m)\big))\to H^*(\map_*\big(\Sigma_{g,1}',S^{2m}_{\Q}\big))$ is the inclusion of algebras 
\begin{equation}
    \chi^*:\Lambda[x_1,...,x_{2g}]\to \Q[y_1,...,y_{2g}]\otimes\Lambda[x_1,...,x_{2g}].
\end{equation}
The map $j$ is the inclusion of the fibre of $\bar{r}$, so, from the spectral sequence of $\bar{r}$, the map $j^*:H^*(F_{\chi\circ r})\to H^*(\Omega^2S^{2m}_{\Q})$ is the algebra surjection
\begin{equation}
    \chi^*:\Q[y_1,...,y_{2g},w]\otimes\Lambda[v]\to
    \Q[w]\otimes\Lambda[v]
\end{equation}
vanishing on the $y_i$. 
Therefore $\chi^*\otimes j^*$ is precisely given by $\phi$ from the statement of the lemma.

Finally, since $\phi$ commutes with the differentials, we have $d_{4m-2}(v)=2\omega$ in $E_*^{*,*}(\chi\circ r)$. There can be no other non-trivial differentials, because $E^{*,*}_2(r)\cong R^{*,*}$ and $E^{*,*}_2(\chi\circ r)\cong Q^{*,*}$ have equal totalisations (see Remark \ref{rmk:comparisonRQ}) and they converge to the same algebra $H^*(\map_{\partial}(\Sigma_{g,1},S^{2m}_{\Q}))$; any more differentials and the dimension of $E^{*,*}_{\infty}(\chi \circ r)$ in some total degree would be too small. Thus, under $E_2^{*,*}\cong Q^{*,*}$, the differential is $d_{4m-3}=d_Q$ as desired. 
\end{proof}

\subsection{Bi-graded algebras and filtered graded algebras}\label{sec:bigradedalgebras}
Our aim here is to state and prove the \textit{lifting lemma} \ref{lem:liftinglemma}.

\begin{definition}
A \textit{filtered algebra} $(A^*,F^*_*)$ is a graded algebra $A^*$ with a filtration 
$$0=F^k_k\subset...\subset F^k_{i+1}\subset F^k_{i}\subset ...\subset F^k_0=A^k,$$
that is multiplicative in the sense that $F^k_i\cdot F^{k'}_{i'}\subset F^{k+k'}_{i+i'}$. The \textit{associated graded} algebra $\gr^F(A)^{*,*}$ is given by the filtration quotients $$\gr^F(A)^{p,q}=F_{q}^{p+q}/F_{q+1}^{p+q}$$ 
and inherits a product from $A^*$. 
A morphism $\phi:(A^*,F^*_*)\to (B^*,G^*_*)$ of graded algebras is \textit{filtered} if $\phi(F^k_i)\subset G^k_i$ for all $0\le i\le k$. In that case, $\phi$ has \textit{associated graded} morphism $$\gr(\phi):\gr^F(A)\to \gr^G(B)$$
which is also a morphism of bigraded algebras. 
\end{definition}

\begin{definition}
For a bigraded algebra $B^{*,*}$, the \textit{tautological filtration} $F_*^*(\Tot)$ of $\Tot(B)^*$ is $$F_{i}\Tot(B)^k =\bigoplus_{p+q=k,q\le i-k}B^{p,q}.$$
\end{definition}
\begin{remark}
Clearly, $\gr^F(\Tot(B))^{*,*}\cong B^{*,*}$ naturally, but in general $\Tot(\gr^F(A))^*$ $\not\cong A^*$.
\end{remark}

For a subset $S\subset \Z_{\ge 0}\times\Z_{\ge 0}$ define $B^S=\oplus_{(p,q)\in S}B^{p,q}$ the bigraded vector subspace of $B^{*,*}$. Write this inclusion as $i_S:B^S\hookrightarrow B^{*,*}$. Clearly if $S$ is additively closed and contains $(0,0)$, then $B^S\subset B^{*,*}$ and $B^S$ is, in fact, a subalgebra. 

\begin{definition}
For $k\ge 0$, we say that $(p,q)\in \Z_{\ge 0}\times\Z_{\ge 0}$ is \textit{a bottom entry in total degree $k$} of $B^{*,*}$ if $p+q=k$ and, for $p'+q'=k$ and $q'<q$, $B^{p',q'}=0$.
\end{definition}

\begin{definition}\label{def:bottomandclosed}
For a filtered algebra $(A^*,F^*_*)$, a set of indices $S\subset\Z_{\ge 0}\times\Z_{\ge 0}$ is \textit{bottom and closed} if it satisfies
\begin{enumerate}
    \item $(0,0)\in S$,
    
    \item each $(p,q)\in S$ is a bottom entry in its total degree for $\gr^F(A)^{*,*}$ and
    
    \item for each $(p,q),(p',q')\in S$, then $(p+p',q+q')$ is also a bottom entry of $\gr^F(A)^{*,*}$ and furthermore, if $\gr^F(A)^{p+p',q+q'}\neq 0$, then $(p+p',q+q')\in S$.
\end{enumerate}
\end{definition}

\begin{lemma}[Lifting Lemma]
Let $(A^*,F^*_*)$ be a filtered graded algebra with a bottom and closed set of indices  $S\subset \Z_{\ge 0}\times\Z_{\ge 0}$. 
Then, $\gr^F(A)^S\subset \gr^F(A)^{*,*}$ is a bigraded subalgebra, and this inclusion of algebras, $i_S$, has a canonical \textit{lift} 
\begin{equation}
    \phi_S: \big(\Tot(\gr^F(A)^S), F^*_*(\Tot)\big)\to (A^*, F^*_*),
\end{equation} 
in the sense that $\phi_S$ is a filtered monomorphism of algebras whose associated graded
\begin{equation}
    \gr(\phi_S): \gr^{F(\Tot)}\big(\Tot(\gr^F(A)^S)\big)=\gr^F(A)^S\to \gr^F(A)^{*,*}    
\end{equation}
is the inclusion $i_S$.
\label{lem:liftinglemma}
\end{lemma}
\begin{proof}
We first construct the linear map $\phi_S$ and prove it is a lift. The idea is that if $(p,q)$ is a bottom entry then $F^{p+q}_{i}=0$ for all $i>q$ and so $\gr^F(A)^{p,q}$ is canonically isomorphic to $F^{p+q}_{q}$. Composing this isomorphism with the inclusion $F^{p+q}_{q}\subset A^{p+q}$, provides a canonical linear monomorphism $\phi_{p,q}:\gr^F(A)^{p,q}\to A^{p,q}$, and the monomorphisms $\phi_{p,q}$ piece together to form the graded linear monomorphism $\phi_S:\Tot(\gr^F(A)^S)\to A^*$. It is then an easy tautological check that $\phi_S$ is filtered with associated graded, $\gr(\phi_S)$, equal to $i_S$ using that $S$ has at most one non-trivial entry in each diagonal, as it contains only bottom entries.

That $\gr^F(A)^S\subset \gr^F(A)^{*,*}$ is a subalgebra follows from Definition \ref{def:bottomandclosed} since $S$ being bottom and closed: from 1, $\gr^F(A)^S$ contains the unit; from 3, it is multiplicatively closed. It remains to prove that $\phi_S$ is multiplicative. 

We use that $S$ consists of bottom entries. For $i=1,2$, pick $(p_i,q_i)\in S$ and $z_i\in F^{p_i+q_i}_{q_i}$ and write $$[z_i]\in \gr^F(A)^{p_i,q_i}= F^{p_i+q_i}_{q_i}/F^{p_i+q_i}_{q_i+1}=F^{p_i+q_i}_{q_i}.$$ 
The last equality holds because $(p_i,q_i)$ is bottom. Observe that $\phi_S([z_i])=\phi_{p_i,q_i}([z_i])=z_i$. Denote by $\cdot_A$ the product in $A$ and $\cdot_{\gr}$ the product in $\gr(A)$. Then by definition, $$[z_1]\cdot_{\gr} [z_2]=[z_1\cdot_Az_2]\in \gr(A)^{p_1+p_2,q_1+q_2}.$$ 
By part 3 of Definition \ref{def:bottomandclosed}, $(p_1+p_2,q_1+q_2)$ is also a bottom entry, so we have two cases: 
\begin{enumerate}
    \item $\gr(A)^{p_1+p_2,q_1+q_2}\neq 0$ and so $(p_1+p_2,q_1+q_2)\in S$, therefore $$\phi_S([z_1]\cdot_{\gr}[z_2])=\phi_S([z_1\cdot_Az_2])=z_1 \cdot_A z_2=\phi_S(z_1)\cdot \phi_S(z_2),$$
    or
    
    \item $\gr(A)^{p_1+p_2,q_1+q_2}=0$, so $F^{p_1+p_2+q_1+q_2}_{q_1+q_2}=0$ and $z_1\cdot_A z_2=0$, and therefore $\phi_S([z_1]\cdot_{\gr}[z_2])=0=\phi_S(z_1)\cdot \phi_S(z_2)$.
\end{enumerate} 
In either case, $\phi_S$ is multiplicative.
\end{proof}

\subsection{Definition and proof of the isomorphism}\label{sec:thehomologyandthemap}
A fibration $f:E\to B$ induces a \textit{Serre filtration} $F^*_*(f)$ of $H^*(E)$ and canonical isomorphisms $\gr^f(H^*(E))\cong E^{*,*}_{\infty}(f)$. From now on, we use the shorthand $$H^*:=H^*(\mapd(\Sigma_{g,1},S^{2m}_{\Q})).$$ Fibrations $r$ and $\chi\circ r$ both have total space $\mapd(\Sigma_{g,1},S^{2m}_{\Q})$ so they define Serre filtrations $F^*_*(r)$ and $F^*_*(\chi\circ r)$ on $H^*$.
The map of fibrations $(\id_M,\chi)$ makes the identity map $\id_M$ of $\mapd(\Sigma_{g,1},S^{2m}_{\Q})$ into a filtered map
\begin{equation}
    \id_{H^*}:(H^*,F^*_*(\chi\circ r))\to (H^*,F^*_*(r)).
\end{equation}
We use these to produce a map $\sigma:H^*(R,d_R)\to H^*$ that will serve as the isomorphism from Theorem \ref{thm:maintheoremwithoutmap}.

We separate the polynomial and exterior part of $R^{*,*}$. The weighted bigraded $\Gamma_{g,1}$-subalgebra
\begin{equation}
(E^{*,*},d_E)=(\Lambda[x_1,...,x_{2g},v],d(v)=2\omega)\subset (R^{*,*},d_R)   
\end{equation}
is the exterior part. Naturally, $(E^{*,*},d_E)$ is also a differential bigraded $\Gamma_{g,1}$-subalgebra of $(Q^{*,*},d_Q)$. The weighted bigraded $\Q$-subalgebra 
\begin{equation}
    P^{*,*}=\Q[y_1,...,y_{2g},w]\subset R^{*,*}
    \label{eq:definitionofP}
\end{equation}
is the polynomial part. We endow algebra $P^{*,*}$ with the obvious $\Gamma_{g,1}$-action (symplectic on the $y_i$ and trivial on $w$) and the trivial differential, so that $(P^{*,*},0)=(R^{*,*},d_R)/(E^{*,*},d_E)$ is an isomorphism of $\Gamma_{g,1}$-algebras. However, the inclusion~\eqref{eq:definitionofP} is not $\Gamma_{g,1}$-equivariant. As $\Q$-algebras, we have
$(R^{*,*},d_R)=(P^{*,*},0)\otimes (E^{*,*},d_E)$
and therefore 
\begin{equation}
    H^{*,*}(R,d_R)=P^{*,*}\otimes H^{*,*}(E,d_E).
    \label{eq:homologydecomposition}
\end{equation} We make this into a $\Gamma_{g,1}$-isomorphism by insisting that the action on the $y_i$ in the right hand side is Johnson. See Figure \ref{fig:algerbasRPE} for a schematic representation.

We describe the differential $d_E$ and the homology $H^{*,*}(E,d_E)$. The $\sptg$-algebra $\Lambda^{*}=\Lambda^{*,0}:=\Lambda[x_1,...,x_{2g}]$, with $|x_i|=(\mu,0)$ has an $\sptg$-equivariant endomorphism $$\Phi:z\mapsto z\wedge \omega, \hspace{3pt}z\in \Lambda^{*}.$$ Denote by $V^{*}:=\Lambda^{*}/(\omega)$ the cokernel $\sptg$-algebra and $K^{*}:=\ker(\Phi)$, a graded $\sptg$-representation. Both $V^{*}$ and $K^{*}$ are naturally $\Lambda^{*}$-modules.

\begin{lemma}[The structure of $H^{*,*}(E,d_E)$]\label{lem:structureofhomology}
The following properties hold:
   \begin{enumerate}
       \item the algebra $E^{*,*}$ is the free $\Lambda^{*}$-module, \begin{equation}
           E^{*,*}=\Lambda^{*,0}\otimes\Q\{1,v\};
       \end{equation}
       
       \item the differential $d_E$ is given by $d_E(z)=0$ and $d_E(zv)=z\wedge\omega=2\Phi(z)$, for $z\in \Lambda^{*}$, and is a $\Lambda^{*}$-module map;
       
       \item the homology $H^{*,*}(E,d_E)$ is the weighted bigraded $\Lambda^{*}$-module
        \begin{equation}
        V^{*,0}\otimes \Q\{1\}\oplus K^{*,0}\otimes\Q\{v\}
        \end{equation}
       
       \item and it is non-zero precisely in the set of bidegrees
    \begin{equation}
        S_E:=\{\big(\mu k,0\big),k=0,1,...,g\}\cup \{\big(\mu k,2\mu-1\big),k=g,g+1,...,2g\}.
    \end{equation}
   \end{enumerate}
\end{lemma}
\begin{proof}
Properties (1) and (2) are immediate and (3) follows from (2). Now, forgetting any gradings, the $\sptg$-equivariant map $\Phi:\Lambda^i(H)\to \Lambda^{i+2}(H)$ given by wedging with $\omega$ is injective for $i\le g-1$ and surjective for $i\ge g-1$ with $\Lambda^i(H)/\im(\Phi)$ a non-zero $\sptg$ irreducible for $i=0,...,g$ (see \cite{FultonHarrisRepTheory2004}). In  particular $\Lambda^i(H)/\im(\Phi)=0$
for $i>g$, and by duality $\Lambda^i(H)/\im(\Phi)\cong \ker^{2g-i}(\Phi)$ for all $i\ge 0$. Remembering the bigradings on the $x_i$, we get that $V^{*}$ is non-zero precisely in bidegrees $(\mu k),k=0,1,...,g$ and $K^{*}$ in bidegrees $(\mu k),k=g,g+1,...,2g$. Using $|v|=(0,2\mu-1)$ and (3), part (4) follows.
\end{proof}

\begin{figure}
    \centering
    \includegraphics[width=0.6\textwidth]{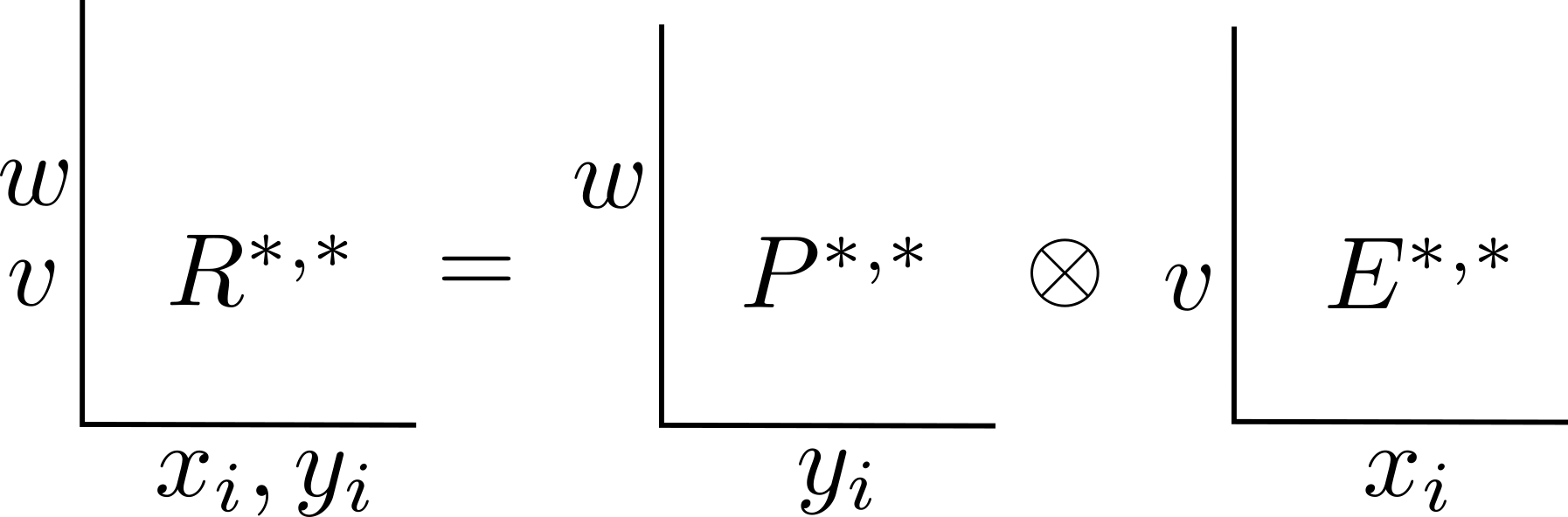}
    \caption{\footnotesize The decomposition of $R^{*,*}$ into polynomial part $P^{*,*}$ and exterior part $E^{*,*}$. The equality is as $\Q$-algebras. To make it $\Gamma_{g,1}$-equivariant, we must choose the Johnson action on the $y_i$ on the right.}
    \label{fig:algerbasRPE}
\end{figure}
\begin{lemma}
     The set of indices $S_E$ is bottom and closed for the bigraded algebra $E^{*,*}_{\infty}(\chi\circ r)$. 
     \label{lem:S_Eisbottom}
\end{lemma}
\begin{proof}
For this argument we write $E^{*,*}_{\infty}:=E^{*,*}_{\infty}(\chi\circ r)$. 
From Lemma \ref{lem:sssofchir}, we have $$\gr^{\chi\circ r}(H^*)=E^{*,*}_{\infty}=\Q\big[y_1[0,2\mu],...,y_{2g}[0,2\mu],w[0,2\mu-1]\big]\otimes H^{*,*}(E,d_E)$$
with all the polynomial generators in the 0-th column. This means that $E^{*,*}_{\infty}$ is the direct sum of copies of $H^{*,*}(E,d_E)$ translated upwards in different rows. This immediately shows $S_E$ satifies the first two conditions of Defintion \ref{def:bottomandclosed} for $E^{*,*}_{\infty}$. It also shows that $E^{p,q}_{\infty}$ vanishes for columns $p> 2g\mu$ so any entry $(p,q)$ with $p\ge 2g\mu$ is automatically bottom. We check the third condition of Defintion \ref{def:bottomandclosed}. 
\begin{enumerate}
    \item For $(i_1,0),(i_2,0)\in \{\big(\mu k,0\big),k=0,1,...,g\}$, the entry $(i_1+i_2,0)$ is bottom by being in the bottom row, and either $i_1+i_2\le g\mu$ so $(i_1+i_2,0)\in S_E$ if or otherwise  $E^{i_1+i_2,0}_{\infty}=0$.
    
    \item For $(i_1,0)\in \{\big(\mu k,0\big),k=0,1,...,g\}$ and $(i_2,2\mu-1)\in\{\big(\mu k,2\mu-1\big),k=g,g+1,...,2g\}$, the entry $(i_1+i_2,2\mu-1)$ is either in $S_E$ if $i_1+i_2\le 2g\mu$, or $i_1+i_2\le 2g\mu$ and so $(i_1+i_2,2\mu-1)$ is bottom and $E^{i_1+i_2,2\mu-1}_{\infty}=0$.
    
    \item For $(i_1,2\mu-1),(i_2,2\mu-1)\in\{\big(\mu k,2\mu-1\big),k=g,g+1,...,2g\}$, $i_1+i_2\ge 2g\mu$ so $(i_1+i_2,4\mu-2)$ is bottom. If $i_1+i_2> 2g\mu$, then $E^{i_1+i_2,4\mu-2}_{\infty}=0$. Otherwise $i_1=i_2=g\mu$,  we also check that $E^{2g\mu,4\mu-2}_{\infty}=0$. Every  $E^{2g\mu,q}_{\infty}$ is a vertical translation of $E^{2g\mu,2\mu-1}_{\infty}$ by a product of $w$ and $y_i$ and any such monomial has bidegree $(0,l(\mu-1)+2k\mu)$, $l,k\in \Z_{\ge 0}$. The equation $2\mu-1+l(\mu-1)+2k\mu=4\mu-2$ has no solutions for $l,k\in \Z_{\ge 0}$.
\end{enumerate}
\end{proof}

\begin{proposition}
     There is a weighted $\Gamma_{g,1}$-equivariant filtered algebra monomorphism
     \begin{equation*}
       \sigma_E:\big(\Tot(H(E,d_E))^*,F^*_*(\Tot)\big)\to \big(H^*,F^*_*(r)\big)
     \end{equation*}
     whose associated graded is the inclusion
     \begin{equation*}
     H^{*,*}(E,d_E)\to H^{*,*}(R,d_R)\cong E_{\infty}^{*,*}(r)
     \end{equation*}
     from the decomposition \eqref{eq:homologydecomposition} and the isomorphism from Lemma \ref{lem:sssofr}.
    \label{prop:sigmaH} 
\end{proposition}
\begin{proof}
Observe that $H^{*,*}(E,d_E)=E^{S_E}_{\infty}(\chi\circ r)$. From Lemma \ref{lem:S_Eisbottom}, $S_E$ is bottom and closed, so we apply the Lifting Lemma \ref{lem:liftinglemma} to deduce that: 1. 
$H^{*,*}(E,d_E)$ is a subalgebra of $E^{*,*}_{\infty}(\chi\circ r)$ and 2. the subalgebra-inclusion $H^{*,*}(E,d_E)\hookrightarrow E^{*,*}_{\infty}(\chi\circ r)$ lifts naturally to a map
$\sigma_H:(H^{*}(E,d_E), F^*_*(\Tot))\to (H^*,F^*_*(\chi\circ r))$.
Finally, we compose with the filtered identity map $\id_{H^*}:(H^*,F^*_*(\chi\circ r))\to (H^*,F^*_*(r))$, and taking associated graded. The second part of Lemma \ref{lem:sssofchir} gives the associated of this map.
\end{proof}

\begin{proposition}\label{prop:sigmaP} 
      There is a weighted filtered algebra monomorphism
     \begin{equation*}
       \sigma_P:\big(\Tot(P)^{*},F^*_*(\Tot)\big)\to \big(H^*,F^*_*(r)\big)
     \end{equation*}
     whose associated graded map is the inclusion
     \begin{equation*}
     P^{*,*}\to E_{\infty}^{*,*}(r)=H^{*,*}(E,d_E)\otimes P^{*,*}.
     \end{equation*}
\end{proposition}
\begin{remark}
The morphism $\sigma_{P}$ is \textit{not} $\Gamma_{g,1}$-equivariant.
\end{remark}
\begin{proof}
The degrees of classes $y_1,...,y_{2g},w$ in $E_{\infty}^{*,*}(r)$ are bottom entries, and so there are canonical lifts of these classes into $y_1,...,y_{2g},w\in H^*$. Extend this to an algebra map $\sigma_P:\Q[y_1,...,y_{2g},w]\to H^*$ which is naturally filtered with respect to $F^*_*(\Tot)$ and $F^*_*(r)$, respectively. The associated graded map is therefore the claimed inclusion, which is a monomorphism. Therefore $\sigma_P$ itself is a monomorphism.
\end{proof}
Finally, by recalling that $R^*=\Tot(R^{*,*})^*$ and  $H^*(R,d_R)=\Tot(H^{*,*}(R,d_R))$, we use the decomposition \eqref{eq:homologydecomposition} to define the algebra morphism
\begin{equation}
    \sigma:=\sigma_P\otimes \sigma_E:\Tot(P)^*\otimes H^*(\Tot(E),d_E)=H^*(R,d_R)\to H^{*}.
\end{equation}
We prove that it has the properties required for Theorem \ref{thm:maintheoremwithoutmap}. 
\begin{proof}[Proof of Theorem \ref{thm:maintheoremwithoutmap}]
Since both $\sigma_P$ and $\sigma_E$ are filtered maps, the tensor product $$\sigma=\sigma_P\otimes \sigma_E:\big(\Tot(H(R,d_R))^*,F^*_*(\Tot)\big)\to \big(H^*,F^*_*(r)\big)$$ is also a filtered map and the associated graded $\gr(\sigma)$ is the tensor product $$\gr(\sigma_P)\otimes\gr(\sigma_E):H^{*,*}(R,d_R)\to E^{*,*}_{\infty}(r).$$ Combining the explicit expressions of $\gr(\sigma_P)$ from Proposition \ref{prop:sigmaH} and $\gr(\sigma_E)$ from Proposition \ref{prop:sigmaP}, and under the isomorphism $E^{*,*}_{\infty}(r)\cong H^{*,*}(R,d_R)$ from Lemma \ref{lem:sssofr}, the map $\gr(\sigma)$ is the identity. In particular, $\sigma$ is an isomorphism. It is weighted because both $\sigma_E$ and $\sigma_P$ are. 

It remains to prove that $\sigma$ is $\Gamma_{g,1}$-equivariant. We do so by exhibiting a generating set of $H^*(R,d_R)$ and checking that $\sigma$ is $\Gamma_{g,1}$-equivariant on the degrees that contain these generators. Obviously, $H^*(R,d_R)$ is generated by the subalgebra $\Tot(H(E,d_E))^*=H^*(\Lambda[x_1,...,x_{2g},v],d(v)=2\omega)$ and the elements $w$ and $y_1,...,y_{2g}$. We check these three cases separately.
\begin{enumerate}
    \item The restriction of $\sigma$ on $\Tot(H(E,d_E))^*\subset\Tot(H(R,d_R))^*$ is $\sigma_{E}$ which is $\Gamma_{g,1}$-equivariant from Proposition \ref{prop:sigmaH}. 
    
    \item The element $w\in H^{\mu-1}(R,d_R)$ has the trivial $\Gamma_{g,1}$ action. Furthermore, $H^{\mu-1}(R,d_R)$ is generated by $w$ as a vector space because all other algebra generators of $R^{*}$ have higher degrees, namely $\mu, 2\mu-1$ and $2\mu$. The Serre spectral sequence of $r$ from Lemma \ref{lem:sssofr} provides a $\Gamma_{g,1}$-equivariant isomorphism $H^{\mu-1}\to H^{\mu-1}(\Omega^2 S^{2m}_{\Q})\cong \Q$, with the trivial action on $\Q$. Thus $\sigma:\Q\langle w\rangle=H^{\mu-1}(R,d_R)\to H^{\mu-1}=\Q$ is (trivially) $\Gamma_{g,1}$-equivariant. 
    
    \item The $y_i$ have degree/weight $(2\mu,(2))$ in $H^*(R,d_R)$. By comparing degrees and weights of the $x_i,y_i,v$ and $w$, we see that the weight $2$ part $H^{2\mu,(2)}(R,d_R)$ is generated by the double wedges $x_ix_j$ and the $y_i$, in particular it is $\Gamma_{g,1}$-isomorphic to $H^{2\mu,0}(R,d_R)=H^{2\mu,0,(2)}(R,d_R)$. The associated graded map $\gr(\sigma)^{2\mu,0}:H^{2\mu,0}(R,d_R)\to E_{\infty}^{2\mu,0}$ is the $\Gamma_{g,1}$-equivariant isomorphism from Lemma \ref{lem:sssofr}.
    Finally, the entry $E_{\infty}^{2\mu,0}$ is bottom and therefore there is a natural, in particular $\Gamma_{g,1}$-equivariant, inclusion $E_{\infty}^{2\mu,0}\subset H^{2\mu}$. Combining these, the restriction of $\sigma$ on degree/weight $(2\mu,(2))$ is given by the $\Gamma_{g,1}$-equivariant composition 
    \begin{equation*}
    \sigma:H^{2\mu,(2)}(R,d_R)=
    \begin{tikzcd}
    H^{2\mu,0}(R,d_R)\rar["\gr(\sigma)^{2\mu,0}","\cong"'] & E_{\infty}^{2\mu,0}\subset H^{2\mu}
    \end{tikzcd}.
    \end{equation*}
    \end{enumerate}
    \end{proof}

\section{Cohomology of Configuration spaces of surfaces as $\Gamma_{g,1}$-representation}\label{sec:conclusion}

\begin{theorem}\label{thm:maintheoremsurfaces}
For $g,i,n\ge 0$, the $\Gamma_{g,1}$-representation $H^i(C_n(\Sigma_{g,1}))$
is isomorphic to the bidegree $(i,(n))$ part of the weighted graded commutative $\Gamma_{g,1}$-algebra
\begin{equation*}
    \Q[y_1,...,y_{2g},w]\otimes H^{*,*}(\Lambda[x_1,...,x_{2g},v], d(v)=2\omega)
\end{equation*}
with
\begin{itemize}
    \item bidegrees $|x_i|=(1,(1))$, $|y_i|=(2,(2))$,
    $|w|=(0,(1))$ and $|v|=(1,(2))$,
    
    \item differential $d$ given on $v$ by 
\begin{equation*}
    d(v)=2\omega=2(x_1x_2+...+x_{2g-1}x_{2g})
\end{equation*}
    and vanishing on the $x_i$, and extended using the Leibniz rule, and  
    
    \item the $\Gamma_{g,1}$-action trivial on $w$ and $v$, symplectic on the $x_i$ and the Johnson on the $y_i$.
\end{itemize}
\end{theorem}
\begin{proof}
From Theorem \ref{thm:maintheoremforgeneralmanifolds}, we have isomorphism 
\begin{equation}
    H^i(C_n(\Sigma_{g,1}))\cong H^{i+2m'n}(\map_{\partial}(\Sigma_{g,1},S^{2(m'+1)}_{\Q}))^{(n)}
\end{equation}
for any $m'\ge 1$.
The right hand side is, by Theorem \ref{thm:maintheoremwithoutmap} and Definition \ref{def:algebraR}, isomorphic to
$$H^{i+2m'n,(n)}\big(R_{m'+1},d_R\big),$$
which, by Remark \ref{lem:allalgebrasSarethesame}, is isomorphic to $H^{i,(n)}\big(R_1,d_R\big)$.  All these isomorphisms are $\Gamma_{g,1}$-equivariant.
Finally, by treating the weighting as a usual grading, the bigraded algebra $H^{*,*}\big(R_1,d_R\big)$ is by definition the weighted graded algebra
$$\Q[y_1,...,y_{2g},w]\otimes H^{*,*}(\Lambda[x_1,...,x_{2g},v], d(v)=2\omega)$$ with the desired $\Gamma_{g,1}$-action.
\end{proof}

\begin{remark}
We can drop the generator $w$, write $A^{i,(\le n)}=\oplus_{j\le n} A^{i,(j)}$ and rephrase the theorem as the isomorphism of $\Gamma_{g,1}$-representations
\begin{equation}
    H^{\bullet}(C_n(\Sigma_{g,1}))\cong\Big(\Q[y_1,...,y_{2g}]\otimes H^{*,*}(\Lambda[x_1,...,x_{2g},v],  d(v)=2\omega)\Big)^{\bullet,(\le n)}
    \label{eq:truncatedringisomorphism}
\end{equation}
\end{remark}
\begin{conjecture}\label{conj:ringstructure} This is an isomorphism of graded $\Gamma_{g,1}$-algebras with the truncated algebra structure on the right.
\end{conjecture}

\subsection{The action of the Johnson filtration}
We determine how the Johnson filtration $J(i)\subset\Gamma_{g,1}$ (see \eqref{eq:JohnsonFiltration}) acts on $H^*(C_n(\Sigma_{g,1}))$.

\begin{proposition}
      The action of $J(2)$ on the representation $J$ is trivial.
\end{proposition}
\begin{proof}
Since $J(2)\subset J(1)=I_{g,1}$, $J(2)$ acts on each of $H,\Lambda^2H$ trivially, therefore by the definition of $J$, it suffices to check that for all $\phi\in J(2)$, $\xi(\phi)(y)$ vanishes for $y\in H$. It suffices to check on the basis $y_i$, $i=1,...,2g$. 

From the definition of $J(2)$, a mapping class $\phi\in J(2)$ maps every $g\in \pi_1(\Sigma_{g,1})$ to $\phi(g)=gh$ for some $h\in \pi_1(\Sigma_{g,1})^{(2)}$.

Then by definition \ref{def:xi}, $\xi(\phi,y_i)=c(\phi(a_i))=c(a_ih)$ some $h\in (\Z^{*2g})^{(2)}$, i.e. a double commutator. From proposition \ref{prop:propertiesofcontent}, we get  $c(a_ih)=c(a_i)+c(h)+[a_i]\wedge [h]$, $c(a_i)=0$ , $c(h)=0$ and $[h]=0$ because $h$ is a commutator.
\end{proof}
\begin{corollary}
The action of $J(2)$ on $H^*(C_n(\Sigma_{g,1}))$ is trivial for all $n\ge 0$ and $g\ge 0$. \label{cor:actionofJ2trivial}
\end{corollary}
\begin{proof}
The vector space $H^i(C_n(\Sigma_{g,1});\Q)$ is a $\Gamma_{g,1}$-subrepresentation of $H^*(R_1,d_R)$ which is a subquotient of the free tensor power on the direct sum of $\Gamma_{g,1}$-representation $J\oplus H\oplus \Q\oplus\Q$ with appropriate gradings (see Remark \ref{rmk:alternativedescriptionofR}) all of which have trivial action of $J(2)$ .
\end{proof}
\begin{remark}
This is in constrast with $J(2)$ acting non-trivially on $H^*(F_n(\Sigma_{g,1}))$ for \textit{ordered} configuration spaces, see \cite{bianchi2021mapping}.
\end{remark}

On the other hand, the $\Gamma_{g,1}$-representation $H^2(C_2(\Sigma_{g,1}))$ was already known to be non-symplectic by \cite{Bianchi2020}. We compute this representation explicitly and conclude this result independently.

\begin{definition}
    The \textit{reduced Johnson representation} is the quotient $\widetilde{J}:=J/\langle \omega \rangle$, fitting in an extension of $\Gamma_{g,1}$-representations
\begin{equation}
    \begin{tikzcd}[column sep=small]
        0\rar &\Lambda^2H/\langle \omega \rangle\rar["i"]& \widetilde{J}\rar & H\rar& 0.
    \end{tikzcd}
    \label{eq:extensionJtilde}
\end{equation}
\end{definition}

\begin{proposition}
If $g\ge 2$, the reduced Johnson representation $\widetilde{J}$ is not symplectic and the extension \eqref{eq:extensionJtilde} is not $\Gamma_{g,1}$-split. In particular, $J$ is also not symplectic.
\label{prop:reducedJnonsymplectic}
\end{proposition}
\begin{proof}
If $J$ were symplectic, then so would be its quotient $\widetilde{J}$. 
If the sequence were $\Gamma_{g,1}$-split, then $\widetilde{J}$ would be a direct sum of symplectic representations, thus symplectic. So it suffices then to prove that $\widetilde{J}$ is not symplectic. It is enough to show that there are $\phi\in I_{g,1}$ and $y\in H$ so that $\xi(\phi,y)\not\in \Q\langle \omega \rangle$. We will prove that the image of $\xi|_{I_{g,1}}:I_{g,1}\to \hom(H,\Lambda^2H)$ is not contained in $\hom(H,\Q\langle \omega \rangle)$. 

We consider the crossed homomorphism $\xi$ with homomorphism $\tau:I_{g,1}\to \hom(H_{\Z},\Lambda^2H_{\Z})$ from \cite{JohnsonAnAbelianQuotient}. Let $a_1,...,a_{2g}\in \pi:=\pi_1(\Sigma_{g,1})$ and $x_1,...,x_{2g}\in H_1(\Sigma_{g,1};\Z)$ be the standard bases and define homomorphism $j:[\pi,\pi]\to \Lambda^2H_{\Z}$ the map $j([a,b])=[a]\wedge [b]$. Then $\tau$ is defined for $\phi\in I_{g,1}$, by  \begin{equation}
    \tau(\phi)(x_i)=j(\phi(a_i)a_i^{-1}).
\end{equation}
Now by definition of $I_{g,1}$, for each $i=1,...,n$ there is a conjugator $c=[a,b]\in [\pi,\pi]$ a conjugator so that $\phi(a_i)=a_ic$. Then  \begin{equation}
    \tau(\phi)(x_i)=j(a_ica_i^{-1})=j([a_i,c]c)=j([a_i,c])+j(c)=[a]\wedge [b]
\end{equation}
since $[c]=0$, being a commutator.
On the other hand, \begin{equation}
    \xi(\phi)(x_i)=c(\phi(a_i))=c(ca_i)=c(c)+c(a_i)+[c]\wedge[a_i]=c(c)=2[a]\wedge [b],
\end{equation} by Proposition \ref{prop:propertiesofcontent}. Extending linearly we conclude $\xi|_{I_{g,1}}=2\tau$. 

Finally, from \cite{JohnsonAnAbelianQuotient}, the image of $\tau$ is $\Lambda^3H\subset H_{\Z}\otimes \Lambda^2H_{\Z}\cong \hom(H_{\Z},\Lambda^2H_{\Z})$. Thus $\im(\xi|_{I_{g,1}})=\Lambda^3H\subset H\otimes \Lambda^2H$, which is not contained in $H\otimes \omega$, provided that $g\ge 2$, so $\dim(\Lambda^3H)={2g \choose 3}>1$.
\end{proof}

\begin{proposition}
      As $\Gamma_{g,1}$ representations, $H^2(C_2(\Sigma_{g,1}))\cong \widetilde{J}$. In particular, $H^2(C_2(\Sigma_{g,1}))$ is not symplectic. 
      \label{prop:torelliactionnontrivial}
\end{proposition}
\begin{proof}
Immediate by inspection of the definition of $R_1$ and $S_1$, and proposition \ref{prop:reducedJnonsymplectic}.
\end{proof}

\subsection{A comment on closed surfaces}\label{sec:closedsurfaces}
Let $\Sigma_g$ be a closed oriented genus $g$ surface and $\Gamma_g$ its mapping class group.
As briefly mentioned in the introduction, in \cite{looijenga2020torelli} it is shown that the $\Gamma_{g}$ representation $H^3(C_3(\Sigma_g))$ is non-symplectic for $g\ge 3$. We briefly explain how the same representation arises from results of our paper. 

The inclusion $\Diffd^+(\Sigma_{g,1})\subset \Diff^+(\Sigma_g)$ gives rise to surjection $\Gamma_{g,1}\to \Gamma_{g}$, and so we can speak of $\Gamma_{g,1}$ representations instead. By applying Theorem~\ref{thm:maintheoremforgeneralmanifolds}, $H^3(C_3(\Sigma_g))$ is isomorphic to the $\Gamma_{g,1}$ representation
\begin{equation}
    \widetilde{H}^{3\mu}(\map(\Sigma_g,S^{2m}_{\Q}))^{(3)}.
\end{equation}
We can compute this by analysing the Serre spectral sequence of the evaluation fibration
\begin{equation}\label{eq:evaluationfibrationclosed}
    \mapd(\Sigma_{g,1},S^{2m}_{\Q})\cong\map_*(\Sigma_g,S^{2m}_{\Q})\to \map(\Sigma_g,S^{2m}_{\Q})\to S^{2m}_{\Q},
\end{equation}
and comparing it with the Serre spectral sequence from Section \ref{sec:freemaps} by deformation retracting $\vee_{2g}S^1$ to $\Sigma_{g,1}$. One of the relevant entries of the Serre spectral sequence of \eqref{eq:evaluationfibrationclosed} is then
\begin{align*}
    E^{0,3\mu, (3)}_{\infty}&=\Q\{x_i\wedge x_j\wedge x_k,x_i y_j:i,j,k=1,...,2g\}/\langle \omega \wedge x_i,x_iy_j=x_jy_i\rangle\\
    &\cong \Lambda^3H/(\omega\wedge H)\oplus \Sym^2H
\end{align*} The direct sum here is not $\Gamma_{g,1}$-equivariant, but rather, the action is again the symplectic on the $x_i$ and the ``Johnson on the $y_i$''. We defer a more detailed elaboration of this to future work.

\bibliographystyle{amsalpha}

\bibliography{biblio}

\end{document}